\documentclass{compositio}
\usepackage{amsmath}
\usepackage{enumerate}
\begin{document}
\numberwithin{equation}{section}
\title[Elliptic curve short average]{Short average distribution of a prime counting function over families of elliptic curves}
\author{Sumit Giri}
\email{sumit.giri199@gmail.com}
\address{School of Mathematics, Tel Aviv University, P.O.B. 39040, Ramat Aviv, Tel Aviv 69978, Israel.}


\newtheorem{theoremm}{Theorem}
\newtheorem{lemma}{Lemma}
\newtheorem{theorem}{Theorem}

\newtheorem{rem}{Remark}
\newtheorem{note}{Note}
\newtheorem{prop}{Proposition}
\newtheorem{cor}{Corollary}
\newtheorem{defn}{Definition}
\newtheorem{conj}{Conjecture}

\renewcommand{\thetheoremm}{\Alph{theoremm}}
\renewcommand{\theenumi}{\Alph{enumi}}
\newcommand{\N}{\mathbb{N}}
\newcommand{\Z}{\mathbb{Z}}
\newcommand{\Q}{\mathbb{Q}}
\newcommand{\R}{\mathbb{R}}
\newcommand{\C}{\mathcal{C}}
\newcommand{\PP}{\mathcal{P}}
\newcommand{\RR}{\mathbb{R}}
\newcommand{\QQ}{\mathbb{Q}}
\newcommand{\CC}{\mathbb{C}}
\newcommand{\NN}{\mathbb{N}}
\newcommand{\ZZ}{\mathbb{Z}}
\newcommand{\FF}{\mathbb{F}}
\newcommand{\m}{\Z /m \Z}
\newcommand{\F}{\mathbb{F}}
\newcommand{\f}{\hat{F}}
\newcommand{\g}{\hat{G}}
\newcommand{\D}{\Z /d \Z}
\newcommand{\n}{\Z /n \Z}
\newcommand{\A}{\mathcal{A}}
\newcommand{\h}{\mathcal{H}}
\newcommand{\B}{\mathcal{B}}
\newcommand{\OO}{\mathcal{O}}
\newcommand{\p}{\textfrak{p}}
\newcommand{\q}{\textfrak{q}}
\newcommand{\bb}{\textfrak{b}}
\newcommand{\aaa}{\textfrak{a}}
\newcommand{\pr}{\displaystyle \prod}
\newcommand{\s}{\displaystyle \sum}

\newcommand{\aut}{\rm Aut}
\newcommand{\FFp}{S,T\in\FF(P)}
\newcommand{\FFps}{S,T\in\FF(P)^*}
\newcommand{\FFpu}{\exists (u_1,\ldots, u_L)\in\FF(P)^*}
\newcommand{\wpab}{w(P, N,E_{a,b})}

\newcommand{\sumpchiimpa}{\underset{\substack{p\leq X \\  p_i^-<p_{i+1}<p_{i}^+ \\ 1 \leq i \leq s-1}}{{\sum }^*}}
\newcommand{\sumpchiimpb}{\underset{\substack{p_i^-<p_{i+1}<p_{i}^+ \\ s+1 \leq i \leq L}}{{\sum }^*}}
\newcommand{\sumpchiprim}{\sumofallps\sum_{\substack{(\chi_i')^6=\chi_0 \imod{p_i}\\
{\rm for}\: 1\leq i\leq \ell\:{\rm and} \\ \exists 1\leq j \leq \ell \:{\rm s.t.}\: \chi_j'\neq \chi_0 \imod{p_j}}} }
\newcommand{\sumofallps}{\sum_{\underset{p_m\neq p_n,\, \forall m\neq n }{\underset{1\le i\le \ell}{N^-<p_i<N^+}}}}
\newcommand{\sumofallh}{\sum_{\underset{p_m\neq p_n,\, \forall m\neq n }{\underset{1\le i\le \ell}{N^-<p_i<N^+}}}\prod_{j=1}^\ell H(D_N(p_j))}
\newcommand{\sumofpfrac}{\sum_{\underset{p_m\neq p_n,\, \forall m\neq n }{\underset{1\le i\le \ell}{N^-<p_i<N^+}}}\frac{1}{p_1^2\cdots p_L^2}}
\newcommand{\sumofpfracsing}{\sum_{\underset{p_m\neq p_n,\, \forall m\neq n }{\underset{1\le i\le \ell}{N^-<p_i<N^+}}}\prod_{i=1}^{\ell}\frac{1}{p_i}}
\newcommand{\sumofpminusone}{\sum_{\underset{p_m\neq p_n,\, \forall m\neq n }{\underset{1\le i\le \ell}{N^-<p_i<N^+}}}\frac{1}{(p_1-1)\cdots (p_\ell-1)}}
\newcommand{\sumofallpsx}{\sum_{\underset{p_m\neq p_n,\, \forall m\neq n }{\underset{1\le i\le e+f}{N^-<p_i<N^+}}}}
\newcommand{\sumofallpsy}{\sum_{\underset{p_m\neq p_n,\, \forall m\neq n }{\underset{e+f+1\le i\le \ell}{N^-<p_i<N^+}}}}
\makeatletter
\def\imod#1{\allowbreak\mkern7mu({\operator@font mod}\,\,#1)}
\makeatother

\begin{abstract}
 Let $E$ be an elliptic curve defined over $\Q$ and let $N$ be a positive integer. Now, $M_E(N)$ counts the number of primes $p$ such that the group $E_p(\F_p)$ is of order $N$. In an earlier joint work with Balasubramanian, we showed that $M_E(N)$ follows Poisson distribution when an average is taken over a family of elliptic curve with parameters $A$ and $B$ where $A,\, B\ge N^{\frac{\ell}{2}}(\log N)^{1+\gamma}$ and $AB>N^{\frac{3\ell}{2}}(\log N)^{2+\gamma}$ for a fixed integer $\ell$ and any $\gamma>0$. In this paper we show that for sufficiently large $N$, the same result holds even if we take $A$ and $B$ in the range $\exp(N^{\frac{\epsilon^2}{20\ell}})\ge A, B>N^\epsilon$ and $AB>N^{\frac{3\ell}{2}}(\log N)^{6+\gamma}$ for any $\epsilon>0$.
\end{abstract}

\maketitle
\section{Introduction}
Let $E$ be an elliptic curve defined over the field of rationals $\Q$ with discriminant $\Delta_E$. For a prime $p$ where $E$ has good reduction, i.e. $p\nmid \Delta_E$, we denote by $E_p$ the reduction of $E$ modulo $p$.
Let $\F_p$ be the finite field with $p$ elements and $E_p(\F_{p})$ be the group of $\F_p$ points over $E_p$. \par
For $p\nmid \Delta_E$, we know $|E_p(\F_p)|=p+1-a_p(E)$ where $a_p(E)$ is the trace of the Frobenius morphism at $p$. By Hasse's theorem we know $|a_p(E)|<2\sqrt{p}$.
For a fixed positive integer $N$, we define the following prime counting function
\begin{align}
 M_E(N):=\#\{p \text{ prime}:\, E\text{ has good reduction over $p$ and } |E_p(\F_p)|=N \}. \label{eq_1.1}
\end{align}
Now, for a pair of integers $(a,b)$, let $E_{a,b}$ be the elliptic curve defined by the Weierstrass equation
$$E_{a,b}:y^2=x^3+ax+b.$$
Also, for $A,B>0$, we define the family of curves $\C(A,B)$  by \begin{align}
                                                               \C(A,B):=\{E_{a,b}:\, |a|\le A, |b|\le B, \Delta(E_{a,b})\neq 0\}.\label{eq_1.6}
                                                              \end{align}
 Now, let us recall Barban-Davenport-Halberstam conjecture for primes in arithmetic progression in short interval.
\begin{conj}(Barban-Davenport-Halberstam)\label{conj_1} Let $\theta(x; q, a) = \underset{p\le x, {p\equiv a (\text{mod }q)}}{\sum} \log p$. Let $0 < \eta \le 1$ and $\beta > 0$ be real numbers. Suppose that $X$, $Y$, and $Q$ are positive real numbers satisfying $X^\eta \le Y \le X$
and $Y /(\log X)^\beta \le Q \le Y$. Then
$$\sum_{q\le Q }\sum_{\underset{(a,q)=1}{1\le a\le q}}|\theta(X + Y ; q, a)- \theta(X; q, a) -\frac{Y}{\phi(q)}|^2\ll_{\eta,\beta} Y Q \log X.$$
 \end{conj}
 Under the above conjecture, David and Smith proved that
\begin{theoremm}\label{thmm_1} Let Conjecture \ref{conj_1} be true for some $0<\eta<\frac{1}{2}$.
If $A, B\ge \sqrt{N}(\log N)^{1+\gamma}\log \log N$ and that $AB\ge N^{\frac{3}{2}}(\log N)^{2+\gamma}\log \log N$, then for any odd integer $N$, we have
\begin{align}
\frac{1}{\#\C(A,B)}\sum_{E\in \C(A,B)}M_E(N)&=\frac{K(N)N}{\phi(N)\log N}+O(\frac{1}{(\log N)^{1+\gamma}}),
\intertext{with }
K(N):=\prod_{p\nmid N}\Bigg( 1-\frac{(\frac{N-1}{p})^2p+1}{(p-1)^2(p+1)}\Bigg)&\prod_{{p\mid N}}\left(1-\frac{1}{p^{\nu_p(N)}(p-1)}\right), \label{eq_1.5} \end{align}
where $\nu_p$ denotes the usual $p$-adic valuation where $\nu_p(n)$ and $\left(\frac{n-1}{p}\right)$ is the Kronecker symbol.\par
\end{theoremm}
Also, by taking another average over $N\le x$, a similar result was unconditionally proven by Chandee, David, Koukoulopoulos and Smith \cite{3}.\par
Improving a result of Martin, Pollack and Smith \cite{9}, in a work with Balasubramanian \cite{10}, we showed that the function $\frac{K(N)N}{\phi(N)}$ is $1$ on an average and the average approaches $1$ reasonably fast. \par

Using an approach first used by Banks and Shparlinski \cite{18}, Balog, Cojocaru and David \cite{1}, Akbary and Felix \cite{19}, in \cite{21} Parks proved that the average result in \emph{Theorem \ref{thmm_1}} is true even if we significantly relax the lower bound conditions on $A$ and $B$. To be precise, he proved
\begin{theoremm}\label{Par}
 Let $\epsilon, \, \gamma >0$ and assume for intervals of length $N^{\eta}$ that the Barban- Davenport-Halberstam Conjecture holds for
\begin{align*}\eta=\frac{1}{2}-(\gamma+2)&\frac{\log \log N}{\log N}.
\intertext{Suppose further that}
\exp({N^{\frac{\epsilon^2}{20}}})\gg A, B>N^{\epsilon} &\text{and } AB>N^{\frac{3}{2}}(\log N)^{6+2\gamma}\log \log N.
\intertext{Then, for any odd integer $N$,}
\frac{1}{\C(A,B)}\sum_{E\in \C(A,B)}M_E(N)&=\frac{K(N)N}{\phi(N)\log N}+O(\frac{1}{(\log N)^{1+\gamma}}),
\end{align*}
where $K(N)$ is given in (\ref{eq_1.5}).
\end{theoremm}

In an earlier work with Balasubramanian \cite{6}, we proved results related to distribution of the function $M_E(N)$. More precisely, we proved that
\begin{theoremm}\label{th_1}
 Let $\C(A,B)$ be as defined as in (\ref{eq_1.6}) and $N$ be a positive integer greater than $7$. If $L$ be a positive integer such that $A,B>N^{L/2}(\log N)^{1+\gamma}$ and $AB>N^{3L/2}(\log N)^{2+\gamma}$ for some $\gamma>0$, then for $1\le \ell\le L-1$
 \begin{align*}
  \frac{1}{\#\C(A,B)}\sum_{\underset{M_E(N)=\ell}{E\in \C(A,B)}}1=\frac{1}{\ell!}\left(\frac{1}{\#\C(A,B)}\sum_{E\in \C(A,B)}M_E(N)\right)^\ell\left(1+O\left(\frac{N}{\phi(N)\log N}\right)\right)+O\left(\frac{1}{N^{\frac{L-\ell}{2}}(\log N)^{\gamma}}\right),
 \end{align*}
where the \textquoteleft$O$\textquoteright constant in the last error term is independent of $\gamma$.
 \end{theoremm}

Using an approach similar to Parks \cite{20,21}, in this paper, we improve \emph{Theorem \ref{th_1}} as follows:

\begin{theorem}\label{Th_01}
 Let $0<\epsilon<1$ be a small positive number and $\ell$ be a positive integer. Suppose $\frac{\log N}{\log \log N}\ge \frac{20\ell}{\epsilon^2}$ and $\exp\left(N^{\frac{\epsilon^2}{20\ell}}\right)\gg A,B>N^\epsilon$ and $AB>N^{\frac{3\ell}{2}}(\log N)^{6+2\gamma}(\log \log N)^\frac{\ell}{2}$, then
\begin{align*}
  \frac{1}{\C(A,B)}\sum_{\underset{M_E(N)=\ell}{E\in \C(A,B)}}1=\frac{1}{\ell!}\left(\frac{1}{\C(A,B)}\sum_{E\in \C(A,B)}M_E(N)\right)^\ell\left(1+O\left(\frac{N}{\phi(N)\log N}\right)\right)+O\left(\frac{1}{(\log N)^{\ell+\gamma}}\right),
 \end{align*}
where the \textquoteleft$O$\textquoteright constant in the last error term is independent of $\gamma$.
 \end{theorem}

Alternatively, under Conjecture \ref{conj_1}, we can state the above theorem in the following form:

\begin{theorem}\label{Th_03}
  Suppose Conjecture \ref{conj_1} be true for some $\eta<\frac{1}{2}$. Let $\gamma_1$ be a non negative integer and $\gamma_2>0$. Also let $\exp\left(N^{\frac{\epsilon^2}{20(\ell+\gamma_1)}}\right)\gg A,B>N^\epsilon$ and $AB>N^{\frac{3(\ell+\gamma_1)}{2}}(\log N)^{6+2\gamma_2}(\log \log N)^\frac{\ell+\gamma_1}{2}$ for a odd positive integer $N$ with $\frac{\log N}{\log \log N}\ge \frac{20(\ell+\gamma_1)}{\epsilon^2}$. Then, for $r\le \ell$
\begin{align*}\frac{1}{\#\C(A,B)}\sum_{E\in \C(A,B)}\sum_{M_E(N)\ge \ell}M_E(N)^r=\sum_{m=\ell}^{\ell+\gamma_1}{d_{\ell,r}(m)}\left(\frac{K(N)N}{\phi(N)\log N}\right)^m+O\left(\frac{N}{\phi(N)\log N}\right)^{1+\ell+\gamma_1}&\\
+O\left(\frac{1}{(\log N)^{\ell+\gamma_2}}\right),&\end{align*}
where $\C(A,B)$ is as before and \begin{align}d_{\ell,r}(m)=\sum_{k=\ell}^m\frac{k^r}{k!}\frac{(-1)^{m-k}}{(m-k)!}.\label{eq_10}\end{align}
\end{theorem}

\textbf{Remark:} Although, in \cite{21}, Theorem \emph{\ref{Par}} is claimed to hold for $\exp(N^\epsilon)\gg A,B>N^\epsilon$, the correct upper bound for $A$ and $B$ should be of the order $\exp(N^{O(\epsilon^2)})$.\par

The crucial difference between proof of Theorem \ref{Th_01} and Theorem \ref{th_1} is Proposition \ref{lm_0.1}, which is stated in Section 3.  In this proposition, we have better estimate of the number of curves of the form $E_{a,b}: y^2=x^3+ax+b$ with $a,b\in \Z$, which simultaneously reduces modulo a given set of distinct primes $(p_1,p_2,\cdots, p_\ell)$ to fixed set of curves of the form $E_{s_1,t_1}/\F_{p_1}$, $E_{s_2,t_2}/\F_{p_2}, \cdots, E_{s_\ell,t_\ell}/\F_{p_\ell}$ for $(s_1,s_2,\cdots, s_\ell), (t_1,t_2,\cdots, t_\ell)\in \F_{p_1}^*\times \cdots \times \F_{p_\ell}^*$\par
In our previous paper with Balasubramanian \cite{6}, we estimated number of curves satisfying above conditions using a technique essentially due to Fouvry and Murty \cite{7}, which involves partitioning a rectangle of size $A\times B$ into boxes of size $p_1p_2\cdots p_\ell\times p_1p_2\cdots p_\ell$ and using Chinese reminder theorem to merge congruence condition over different primes together. While in Proposition \ref{lm_0.1} we use estimates of sums of suitable multiplicative characters.

\textbf{Acknowledgements: }I would like to thank Chantal David, Dimitris Koukoulopoulos and Amir Akbary for their useful advices. I also thank James Parks for spending time in clarifying some of my doubts that were essential for completing the work.

\textbf{Funding: }This work was done while working as a Postdoctoral Fellow at Centre de Recherches Mathematiques, Université de Montreal.
\section{Preliminaries}

Let $D$ be a negative discriminant. Using the class number formula [p. 515, \cite{14}], the \emph{Kronecker class number} for a discriminant $D$ can be written as
\begin{align}
 H(D):=\sum_{\underset{D/f^2\equiv 0,1\, (\text{mod }4)}{f^2\mid D}}\frac{\sqrt{|D|}}{2\pi f}L(1, \chi_{D/f^2})\label{eq_2.1}
\end{align}
where $\chi_d$ is the Kronecker symbol $(\frac{d}{.})$ and $L(s,\chi_d):=\overset{\infty}{\underset{n=1}{\sum}}\frac{\chi_d(n)}{n^s}$.\par

Using Deuring's theorem \cite{16} we get
   \begin{align}
   H(t^2-4p)&=\sum_{\underset{a_p(E)=t}{\tilde{E}/\F_p}}\frac{1}{\#\text{Aut}(\tilde{E})},\label{eq_deu}
   \end{align}

where the sum is over the $F_p$-isomorphism classes of elliptic curves.\par
Define, \begin{align}
                  D_N(p):=(p+1-N)^2&-4p= (N+1-p)^2-4N,\label{eq_2.2}\\
                  N^+:=(\sqrt{N}+1)^2 \,; &\quad N^-:=(\sqrt{N}-1)^2\nonumber\\
                  d_{N,f}(p):=\frac{D_N(p)}{f^2} &\text{ for } f^2\mid D_N(p). \nonumber
                   \end{align}
With these notations defined, we recall the following lemma from [Lemma 2.1, \cite{6}]
\begin{lemma}\label{lm_1.1}
Let $N$ be a positive integers and $N^-$ and $N^+$ are defined as above. Also
 let $H(D_N(p))$ be defined by (\ref{eq_2.1}) and (\ref{eq_2.2}). Then
\begin{enumerate}[(a)]
\item $$\underset{N^-<p<N^+}{\sum}{H(D_N(p))} \ll  \frac{N^2}{\phi(N)\log N}.$$
 \item For $k\ge 2$, $$\underset{N^-<p<N^+}{\sum}{H(D_N(p))^k} \ll  N^{\frac{k+1}{2}}(\log N)^{k-2}(\log \log N)^k.$$
\end{enumerate}
\end{lemma}
We also need the following two theorems:

\begin{theorem} \label{largesieve}
Let $M, N, Q$ be positive integers and let $\left\{a_n\right\}_n$ is a sequence of complex numbers. For a fixed $q \leq Q$, we let
$\chi$ be a Dirichlet character modulo $q$. Then
$$\sum_{q\leq Q}\frac{q}{\phi(q)}\sum_{\substack{\chi \imod{q} \\ \chi\:{\rm primitive}}}\left|\sum_{M< n\leq M+N} a_n\chi(n)\right|^2
 \leq (N+3Q^2)\sum_{M< n\leq M+N}|a_n|^2.$$
\end{theorem}
For the proof of the above theorem, see [Chapter 27, \cite{12}].\\
The second theorem is due to Friedlander and Iwaniec \cite{JFHI:2}, which bounds the fourth power moment of Dirichlet characters.

\begin{theorem}(\textbf{Friedlander-Iwaniec}) \label{fourthpower}
Let $q$ and $N$ be positive integers. Let $\chi$ denote a Dirichlet character modulo $q$, with $\chi_0$ denoting the principal character.
Then $$ \sum_{\chi\neq \chi_0}\left|\sum_{n\leq N}\chi(n)\right|^4 \ll N^2q\log^6 q.$$
\end{theorem}

\section{Proof of Theorems}
Let $r\ge 1$ be a positive integer.
We have,
\begin{align}
\frac{1}{\#\C(A,B)}\sum_{\underset{M_E(N)\ge \ell}{E\in \C(A,B)}}M_E(N)^r &= \frac{1}{\#\C(A,B)}\sum_{\underset{M_E(N)\ge \ell}{E\in \C(A,B)}}\left(\sum_{\underset{E_p(\F_p)=N}{N^-<p<N^+}}1 \right)^r\nonumber\\
&= \frac{1}{\#\C(A,B)}\sum_{N^-<p_1,\cdots, p_r<N^+}\sum_{\underset{E_{p_1}(\F_{p_1})=\cdots=E_{p_r}(\F_{p_r})=N}{E\in \C(A,B), \, M_E(N)\ge \ell}}1.\nonumber
\end{align}
For any non-negative integer $\gamma_1$, breaking the sum into two parts, the right hand side can be written as

\begin{align}
\frac{1}{\#\C(A,B)}\sum_{\underset{1\le i\le r}{N^-<p_i<N^+}}&\sum_{j=\ell}^{\ell+\gamma_1}\sum_{\underset{\underset{E\in \C(A,B),\,  1\le i\le r}{E_{p_i}(\F_{p_i})=N}}{ M_E(N)= j}}1+ \frac{1}{\#\C(A,B)}\sum_{\underset{1\le i\le r}{N^-<p_i<N^+}}\sum_{\underset{\underset{E\in \C(A,B), \, 1\le i\le r}{E_{p_i}(\F_{p_i})=N}}{ M_E(N)\ge \ell+\gamma_1+1}}1\label{eq_24}
\end{align}
For $r\le {\ell}$, consider the expression
 \begin{align}\label{B_1}
  \frac{1}{\#\C(A,B)}\sum_{\underset{1\le i\le r}{N^-<p_i<N^+}}\sum_{\underset{\underset{\underset{ 1\le i\le r}{E_{p_i}(\F_{p_i})=N}}{ M_E(N)\ge {\ell}+1}}{E\in \C(A,B)}}1
 \end{align}

Now, for a curve $E$ with $M_E(N)=L\ge \ell+1$, the curve $E$ is counted $L^r$ times in (\ref{B_1}). While, the same $E$ is counted $\frac{L!}{(L-{\ell}-1)!}$ times if we consider the expression
\begin{align}\frac{1}{\#\C(A,B)}\sum_{\underset{\underset{p_m\neq p_n \, \text{for }m\neq n}{1\le i\le {\ell}+1}}{N^-<p_i<N^+}}\sum_{\underset{\underset{1\le i\le {\ell}+1}{E_{p_i}(\F_{p_i})=N}}{E\in \C(A,B)}}1\label{eq_008}\end{align}

Using Stirling's approximation, is easy to see that $\frac{L^r(L-{\ell}-1)!}{L!}\ll e^{{\ell}}$ for $r\le {\ell}$.
Thus \begin{align}\label{B_2}
\frac{1}{\#\C(A,B)}\sum_{\underset{1\le i\le r}{N^-<p_i<N^+}}\sum_{\underset{\underset{\underset{ 1\le i\le r}{E_{p_i}(\F_{p_i})=N}}{ M_E(N)\ge \ell+\gamma_1+1}}{E\in \C(A,B)}}1
 &\ll_{\ell, \gamma_1} \frac{1}{\#\C(A,B)}\sum_{\underset{p_m\neq p_n,\, \forall m\neq n }{\overset{N^-<p_i<N^+}{1\le i\le \ell+\gamma_1+1 }}}\sum_{\underset{\underset{1\le i\le {\ell}+\gamma_1+1}{E_{p_i}(\F_{p_i})=N}}{E\in \C(A,B)}}1
 \end{align}
For $r\le \ell\le j\le \ell+\gamma_1$, using a similar argument, one can also show that
\begin{align}
 \frac{1}{\#\C(A,B)}\sum_{\underset{1\le i\le r}{N^-<p_i<N^+}}\sum_{\underset{\underset{E\in \C(A,B),\,  1\le i\le r}{E_{p_i}(\F_{p_i})=N}}{ M_E(N)= j}}1=\frac{j^r}{j!}\frac{1}{\#\C(A,B)}\sum_{\underset{\underset{p_m\neq p_n \, \text{for }m\neq n}{1\le i\le j}}{N^-<p_i<N^+}}\sum_{\underset{\underset{1\le i\le j, \, M_E(N)=j}{E_{p_i}(\F_{p_i})=N}}{E\in \C(A,B)}}1
\end{align}

Also, for $r\le\ell\le j\le \ell+\gamma_1$,
\begin{align}
\sum_{\underset{\underset{p_m\neq p_n \, \text{for }m\neq n}{1\le i\le r}}{N^-<p_i<N^+}}\sum_{\underset{\underset{1\le i\le r, \, M_E(N)=j}{E_{p_i}(\F_{p_i})=N}}{E\in \C(A,B)}}1&=\frac{1}{(j-r)!}\sum_{\underset{\underset{p_m\neq p_n \, \text{for }m\neq n}{1\le i\le j}}{N^-<p_i<N^+}}\sum_{\underset{\underset{1\le i\le j, \, M_E(N)=j}{E_{p_i}(\F_{p_i})=N}}{E\in \C(A,B)}}1\label{eq_27}
\end{align}
 We now consider the first term of (\ref{eq_24}). Note that, the primes in the range of summations in (\ref{eq_24}) are not distinct.
Recalling the definition of $S(n,m)$, the Stirling number of the second kind, which equals to the number of ways of partitioning a set of $n$ elements into $m$ non empty sets, we get

\begin{align}
\sum_{\underset{1\le i\le r}{N^-<p_i<N^+}}\sum_{\underset{E(\F_{p_1})=\cdots=E(\F_{p_r})=N}{E\in \C, \, M_E(N)= j}}1
 &=\left( \sum_{m=1}^r\frac{S(r,m)}{(j-m)!}\right)\sum_{\underset{1\le i\le r}{N^-<p_i<N^+}}\sum_{\underset{E_{p_1}(\F_{p_1})=\cdots=E_{p_r}(\F_{p_r})=N}{E\in \C(A,B), \, M_E(N)= j}}1.\label{eq_2.7}
\end{align}
 To simplify the first factor on the right hand side, we use the equality  $ \sum_{m=1}^r\frac{S(r,m)j!}{(j-m)!}=j^r$. See [(4.1.3), p. 60 , \cite{13}].\par

 With this,
\begin{align}
&\sum_{N^-<p_1\neq p_2\neq \cdots \neq p_j<N^+}\sum_{\underset{E_{p_1}(\F_{p_1})=\cdots=E_{p_r}(\F_{p_r})=N}{E\in \C(A,B), \, M_E(N)= j}}1 \nonumber\\
&=\sum_{N^-<p_1\neq p_2\neq \cdots \neq p_j<N^+}\sum_{\underset{E(\F_{p_1})=\cdots=E(\F_{p_j})=N}{E\in \C(A,B), \, M_E(N)\ge j}}1
                     -\sum_{N^-<p_1\neq p_2\neq \cdots \neq p_j<N^+}\sum_{\underset{E(\F_{p_1})=\cdots=E(\F_{p_j})=N}{E\in \C(A,B), \, M_E(N)\ge j+1}}1\label{eq_19}
\end{align}
Now we denote the left hand side of (\ref{eq_27}) by $\omega(r,j)$ and the first term of the right hand side of (\ref{eq_19}) by $\Omega(j,j)$. Also we call the left hand side of (\ref{eq_2.7}) by $\Upsilon(r,j)$. Then, in view of (\ref{eq_27}) and (\ref{eq_2.7}), we get the following set of relations
\begin{align}
\left\{ \begin{array}{ll}
\Upsilon(r,j)=\frac{j^r}{j!}\omega(j,j),\\
\Omega(t,s)=\overset{\infty}{\underset{n=s}{\sum}} \omega(t,n) \quad \text{for }t\le s,\\
\omega(t,n)=\frac{1}{(n-t)!}\omega(n,n) \quad \text{for }t\le n.
 \end{array}
\right.\label{eq_4}
\end{align}
Now, we state the following Proposition, whose proof will be completed in Section 4.
\begin{prop}\label{lm_0.1}
Let $\C(A,B)$ be as above. Let $0<\epsilon<1$ be a small positive number. Suppose $N$ be a positive integer such that $\frac{\log N}{\log \log N}\ge \frac{20\ell}{\epsilon^2}$ with $\exp\left(\left(\frac{N}{\log N}\right)^{\frac{\epsilon^2}{20\ell}}\right)\gg A,B>N^\epsilon$ and $AB>N^{\frac{3\ell}{2}}(\log N)^{6+2\gamma_2}(\log \log N)^\frac{\ell}{2}$, then
 \begin{align*}\frac{1}{\#\C(A,B)}\sum_{\underset{\underset{p_m\neq p_n \, \text{for }m\neq n}{1\le i\le \ell}}{N^-<p_i<N^+}}\sum_{\underset{\underset{1\le i\le \ell}{E_{p_i}(\F_{p_i})=N}}{E\in \C(A,B)}}1= &\left( \sum_{N^-<p< N^+}\frac{H(D_N(p))}{p} \right)^\ell+O\left(\frac{1}{(\log N)^{\ell+\gamma_2}}\right).
 \end{align*}
\end{prop}

Now, by Proposition \ref{lm_0.1}
\begin{align*}
 \frac{1}{\#\C(A,B)}\Omega(j,j)=\Bigg(\sum_{N^-<p<N^+}\frac{H(D_N(p))}{p}\Bigg)^j+O\Big(\frac{1}{(\log N)^{j+\gamma_2}}\Big),
\end{align*}
whenever $\exp\left(\left(\frac{N}{\log N}\right)^{\frac{\epsilon^2}{20\ell}}\right)\gg A,B>N^\epsilon$ and $AB>N^{\frac{3\ell}{2}}(\log N)^{6+2\gamma_2}$.\par

 Now, we replace $\overset{\ell+\gamma_1}{\underset{j=\ell}{\sum}}\Upsilon(r,j)$ by $\overset{\ell+\gamma_1}{\underset{j=\ell}{\sum}}z_{\ell,r}(j)\Omega(j,j)+O(\Omega(\ell+\gamma_1,\ell+\gamma_1+1))$ where $\{z_{\ell,r}(j)\}$ are some constants to be determined later using (\ref{eq_4}).
Also note that $\Omega(\ell+\gamma_1,\ell+\gamma_1+1)\ll AB\left[\left(\sum_p\frac{H(D_N(p))}{p}\right)^{\ell+\gamma_1}+\frac{1}{(\log N)^{\ell+\gamma_2}}\right]$.\par
Then, in view of (\ref{B_2}) and Proposition \ref{lm_0.1}, the expression in (\ref{eq_24}) equals to
\begin{align}
 \sum_{j=\ell}^{\ell+\gamma_1}z_{\ell,r}(j)\left(\sum_{N^-<p<N^+}\frac{H(D_N(p))}{p}\right)^j+O\left(\sum_{N^-<p<N^+}\frac{H(D_N(p))}{p}\right)^{\ell+\gamma_1+1}+O\left(\frac{1}{(\log N)^{\ell+\gamma_2}}\right)\label{eq_009}
\end{align}
for some real numbers $\{z_{\ell,r}(j)\}_{j=\ell}^{\ell+\gamma_1}$\\
Only thing that remains to be shown is that $\{z_{\ell,r}(j)\}_j$ are equals to $\{d_{\ell,r}(j)\}_j$, as defined in (\ref{eq_10}). For that, we have the following lemma.\par

\begin{lemma}\label{lm_2}
Consider $\omega,\, \Omega$ as variables satisfying the identities in (\ref{eq_4}). Then, the solution of the equation
$$\sum_{j=\ell}^\infty \frac{j^r}{j!}\omega(j,j)=\sum_{j=\ell}^{\infty}z_{\ell,r}(j)\Omega(j,j)$$
in variables $z_{\ell,r}(j)$ are given by $$z_{\ell,r}(j)=\sum_{k=\ell}^j\frac{k^r}{k!}\frac{(-1)^{j-k}}{(j-k)!}=d_{\ell,r}(j).$$
\end{lemma}
\begin{proof}
See [Lemma 3.2, \cite{6}] for the proof of the above lemma.
\end{proof}

Finally, combining (\ref{B_1}), (\ref{eq_009}) and Lemma \ref{lm_2}, we have
\begin{align}
 \frac{1}{\#\C(A,B)}\sum_{\underset{M_E(N)\ge \ell}{E\in \C(A,B)}}M_E(N)^r=& \sum_{j=\ell}^{\ell+\gamma_1}d_{\ell,r}(j)\left(\sum_{N^-<p<N^+}\frac{H(D_N(p))}{p}\right)^j\nonumber\\&+O\left(\sum_{N^-<p<N^+}\frac{H(D_N(p))}{p}\right)^{\ell+\gamma_1+1}+O\left(\frac{1}{(\log N)^{\ell+\gamma_2}}\right)\label{eq_11}
\end{align}
for $\exp\left(\left(\frac{N}{\log N}\right)^{\frac{\epsilon^2}{20(\ell+\gamma_1)}}\right)\gg A,B>N^\epsilon$ and $AB>N^{\frac{3(\ell+\gamma_1)}{2}}(\log N)^{6+\gamma_2}$.

Putting $\ell=1$, $r=1$ and $\gamma_1=0$, $\gamma_2=\gamma$, from (\ref{eq_11}) we get,
\begin{align}
\frac{1}{\#\C(A,B)}\sum_{E\in \C(A,B)}M_E(N)=\sum_{N^-<p<N^+}\frac{H(D_N(p))}{p}+O\left(\left(\sum_{N^-<p<N^+}\frac{H(D_N(p))}{p}\right)^2\right)\nonumber\\
+O\Big(\frac{1}{(\log N)^{1+\gamma}}\Big)& \label{eq_7}
\end{align}
for $\exp\left(\left(\frac{N}{\log N}\right)^{\frac{\epsilon^2}{20}}\right)\gg A,B>N^\epsilon$ and $AB>N^{\frac{3}{2}}(\log N)^{6+\gamma}$\\

Also, for $\gamma_1=0$, $\gamma_2=\gamma$, from (\ref{eq_11}) we have
\begin{align}
 \frac{1}{\#\C(A,B)}\sum_{\underset{M_E(N)= \ell}{E\in \C(A,B)}}M_E(N)^r= d_{\ell,r}(\ell)\left(\sum_{N^-<p<N^+}\frac{H(D_N(p))}{p}\right)^\ell&+O\left(\sum_{N^-<p<N^+}\frac{H(D_N(p))}{p}\right)^{\ell+1}\nonumber\\&+O\left(\frac{1}{(\log N)^{\ell+\gamma}}\right)\nonumber
 \intertext{or,}
 \frac{1}{\#\C(A,B)}\sum_{\underset{M_E(N)= \ell}{E\in \C(A,B)}}1= \frac{d_{\ell,r}}{\ell^r}(\ell)\left(\sum_{N^-<p<N^+}\frac{H(D_N(p))}{p}\right)^\ell&+O\left(\sum_{N^-<p<N^+}\frac{H(D_N(p))}{p}\right)^{\ell+1}\nonumber\\&+O\left(\frac{1}{(\log N)^{\ell+\gamma}}\right)
 \label{eq_12}
\end{align}
for $\exp\left(\left(\frac{N}{\log N}\right)^{\frac{\epsilon^2}{20\ell}}\right)\gg A,B>N^\epsilon$ and $AB>N^{\frac{3\ell}{2}}(\log N)^{6+\gamma}$.\\
We use (\ref{eq_7}) and (\ref{eq_12}) to replace $\underset{N^-<p<N^+}{\sum}\frac{H(D_N(p))}{p}$ in the right hand side of (\ref{eq_12}) by $\frac{1}{\#\C(A,B)}\sum_{E\in \C(A,B)}M_E(N)$.
Now, using Lemma \ref{lm_1.1}a,  we get
\begin{align*} \frac{1}{\#\C(A,B)}\sum_{\underset{M_E(N)= \ell}{E\in \C(A,B)}}1= \frac{d_{\ell,r}}{\ell^r}(\ell)\left( \frac{1}{\#\C(A,B)}\sum_{E\in \C(A,B)}M_E(N) \right)^\ell&\left(1+O\left(\frac{N}{\varphi(N)\log N}\right)\right)\nonumber\\&+O\left(\frac{1}{(\log N)^{\ell+\gamma}}\right)
 \end{align*}
for $\exp\left(\left(\frac{N}{\log N}\right)^{\frac{\epsilon^2}{20\ell}}\right)\gg A,B>N^\epsilon$ and $AB>N^{\frac{3\ell}{2}}(\log N)^{6+\gamma}$.
Further, we recall that $d_{\ell,r}(\ell)=\frac{\ell^r}{\ell!}$. This proves Theorem \ref{Th_01}.\par
\medskip Assuming Conjecture \ref{conj_1} and the proof of [Theorem 3, \cite{4}], we have that
\begin{align}
 \sum_{N^-<p<N^+}\frac{H(D_N(p))}{p}=\frac{K(N)N}{\varphi(N)\log N}+O\Big(\frac{1}{(\log N)^{1+\gamma}}\Big)\label{eq_13}
\end{align}
for odd integer $N$, where $K(N)$ is given by (\ref{eq_1.5}). Combining (\ref{eq_13}) with (\ref{eq_11}), we complete the proof of Theorem \ref{Th_03}.

\section{Proof of Proposition \ref{lm_0.1}}
Before proceeding with the proof of the proposition, we define some standard notations. \\
Let $P:=(p_1,\ldots,p_\ell)$ be a vector of $\ell$ distinct primes such that $(\sqrt{N}-1)^2<p_i<(\sqrt{N}+1)^2$ for $1\le i\le \ell$. So, the primes in question are effectively of the order $N$.

Let $S:=(s_1,\ldots,s_\ell)$ and $T:=(t_1,\ldots,t_\ell)$ be elements of $\F_{p_1}^*\times \F_{p_2}^*\times\cdots \times \F_{p_\ell}^*$. For such $S$ and $T$ as above, define the following indicator function
\begin{equation} \label{defnofwpst}
h(P,N,S,T):=\begin{cases} 1 & {\rm if}\: \#E_{p_i,s_i,t_i}(\FF_{p_i})=N\:{\rm for}\:1\leq i\leq \ell, \\ 0 & {\rm otherwise}. \end{cases}
\end{equation}

Also,
\begin{equation*}
\sum_{\FFp}1=\sum_{\substack{1\leq s_1\leq p_1 \\ 1\leq t_1\leq p_1}}\cdots\sum_{\substack{1\leq s_\ell\leq p_\ell \\
1\leq t_\ell\leq p_\ell}}1 \quad{\rm and} \quad \sum_{\FFps}1=\sum_{\substack{1\leq s_1< p_1 \\
1\leq t_1< p_1}}\cdots\sum_{\substack{1\leq s_L< p_\ell \\ 1\leq t_L< p_\ell}}1.
\end{equation*}

Throughout the rest of the proof, $E_{s,t}:\,y^2=x^3+sx+t$ denotes a curve over a finite field $\F_p$. Also $E_{a,b}$ denotes curve over $\mathbb{Q}$ as defined in (\ref{eq_1.6}).\\
Further, we know, two elliptic curves $E_{s,t}$ and $E_{s',t'}$ are isomorphic over $\FF_{p}$ if and only if there exists a $u\in \FF_{p}^*$ such that $s'=su^4$ and $t'=tu^6$.
Hence, the number of elliptic curves over $\FF_{p}$ isomorphic to $E_{s,t}$ is
$$\frac{\#\FF_{p}^*}{\#\aut(E_{s,t})}= \frac{p-1}{\#\aut(E_{s,t})}.$$
Further,
\begin{equation*}
\#\aut(E_{s,t})=\begin{cases} 6 & {\rm if}\: s=0\:{\rm and}\: p\equiv 1 \imod{3}, \\
4 & {\rm if}\: t=0\:{\rm and}\: p\equiv 1 \imod{4}, \\
2 & {\rm otherwise}.
\end{cases}
\end{equation*}

For $1\leq i\leq \ell$,
\begin{align}
\sum_{\substack{1\leq s_i,t_i\leq p_i \\ \#E_{p_i,s_i,t_i}(\FF_{p_i})={N}}}1&=\sum_{\substack{\bar{E}_{p_i,s_i,t_i}/ \FF_{p_i} \\ p_i+1-a_{p_i}(\bar{E}_{p_i,s_i,t_i})=N}}\frac{p_i-1}{\#\aut(\bar{E}_{p_i,s_i,t_i}(\FF_{p_i}))}, \label{alitrivtwo}
\end{align}
where the summation in the right hand side of (\ref{alitrivtwo}) runs over isomorphism classes of elliptic curve $\bar{E}_{p_i,s_i,t_i}(\FF_{p_i})$. Further, using (\ref{eq_deu}), from (\ref{alitrivtwo}) we get

\begin{align}
\sum_{\substack{1\leq s_i,t_i\leq p_i \\ \#E_{p_i,s_i,t_i}(\FF_{p_i})={N}}}1&=(p_i-1)H(D_N(p_i))\label{eq_001}
\end{align}
Now, the left hand side of the proposition \ref{lm_0.1} is equal to
\begin{align}
\frac{1}{\#\C(A,B)}\sum_{\underset{p_m\neq p_n,\, \forall m\neq n }{\underset{1\le i\le \ell}{N^-<p_i<N^+}}}&\sum_{\underset{\underset{1\le i\le \ell}{E_{p_i}(\F_{p_i})=N}}{E\in \C(A,B)}}1\nonumber\\
&=\frac{1}{\#\C(A,B)}\sum_{\underset{p_m\neq p_n\, \forall m\neq n }{\underset{1\le i\le \ell}{N^-<p_i<N^+}}}\sum_{\FFp}h(P,N,S,T)\sum_{\substack{|a|\leq A, |b|\leq B \\ a\equiv s_i \imod{p_i} \\ b\equiv t_i \imod{p_i} \\ 1\leq i \leq \ell}}1.\label{eq_010}
\end{align}

We plan to count the number of curves $E_{a,b}\in \C(A,B)$ whose reductions modulo $p_i$ are $E_{s_i,t_i}$ for all $p_i$. Then, the inner summation on left hand side of (\ref{eq_010}) can be written as
\begin{align}
\frac{1}{\#\C(A,B)}\sum_{\underset{\underset{1\le i\le \ell}{E_{p_i}(\F_{p_i})=N}}{E\in \C(A,B)}}1 &=\frac{1}{\#\C(A,B)} \sum_{\FFp}h(P,N,S,T)\prod_{j=1}^\ell\frac{\#\aut(E_{p_j,s_j,t_j})}{(p_j-1)}
\sum_{\substack{|a|\leq A, |b|\leq B \\ \exists (u_1,\ldots,u_\ell) \in \FF(P)^*\\ a \equiv s_iu_i^4 \imod{p_i},\\
b \equiv t_iu_i^6 \imod{p_i} \\ \forall 1\leq i \leq \ell}}1.\nonumber\\
&=\frac{1}{\#\C(A,B)}\sum_{\underset{p_m\neq p_n,\, \forall m\neq n }{\underset{1\le i\le \ell}{N^-<p_i<N^+}}}\sum_{\FFp}h(P,N,S,T)Z(P,S,T)
\prod_{j=1}^\ell\frac{\#\aut(E_{p_j,s_j,t_j})}{(p_j-1)}, \label{aliquotaut}
\end{align}
where $Z(P,S,T)$ denotes the number of integers $|a|\leq A,|b|\leq B$ such that $\exists$ $(u_1,\ldots,u_\ell) \in \FF(P)^*$ such that
\begin{equation}
 a\equiv s_iu_i^4\imod{p_i}, \quad b\equiv t_iu_i^6\imod{p_i}
\quad{\rm for}\:1\leq i\leq \ell.\nonumber
\end{equation}

Now, $\#\aut(E_{s,t})=2$ most of the times and in particular when $st\neq 0$. So, we write (\ref{aliquotaut}) as
\begin{align}
\frac{2^\ell}{\#\C(A,B)}\sum_{\underset{p_m\neq p_n,\, \forall m\neq n }{\underset{1\le i\le \ell}{N^-<p_i<N^+}}}&\sum_{\FFps}
\frac{h(P,N,S,T)Z(P,S,T)}{(p_1-1)\cdots(p_\ell-1)}
\nonumber \\
+&\frac{1}{\#\C(A,B)}\sum_{\underset{p_m\neq p_n,\, \forall m\neq n }{\underset{1\le i\le \ell}{N^-<p_i<N^+}}}\sum_{\substack{\FFp \\ s_it_i= 0 \\{\rm for\:some}\: 1\leq i \leq \ell}}h(P,N,S,T)Z(P,S,T)
\prod_{j=1}^\ell\frac{\#\aut(E_{p_j,s_j,t_j})}{(p_j-1)}.\label{alione}
\end{align}
Define \begin{align}
\Sigma_1&:=\frac{2^\ell}{\#\C(A,B)}\sum_{\underset{p_m\neq p_n,\, \forall m\neq n }{\underset{1\le i\le \ell}{N^-<p_i<N^+}}}\sum_{\FFps}
\frac{h(P,N,S,T)Z(P,S,T)}{(p_1-1)\cdots(p_\ell-1)}\label{eq_004}\\
\Sigma_2&:=\frac{1}{\#\C(A,B)}\sum_{\underset{p_m\neq p_n,\, \forall m\neq n }{\underset{1\le i\le \ell}{N^-<p_i<N^+}}}\sum_{\substack{\FFp \\ s_it_i= 0 \\{\rm for\:some}\: 1\leq i \leq \ell}}h(P,N,S,T)Z(P,S,T)
\prod_{j=1}^\ell\frac{\#\aut(E_{p_j,s_j,t_j})}{(p_j-1)}.\label{eq_003}
\end{align}
We plan to complete the estimation of $\Sigma_1$ first. Later, we show that the same estimation technique can be modified suitably to give required upper bound to $\Sigma_2$.\\
For this part of the proof related to the estimation of $\Sigma_1$, we are essentially going to follow the approach of Parks \cite{20}, except possibly the different range of summation over primes.

Separating the expected main term from the expected error term in $\Sigma_1$, we write
\begin{align}
\Sigma_1:=&\frac{4AB}{\#\C(A,B)}\sum_{\underset{p_m\neq p_n,\, \forall m\neq n }{\underset{1\le i\le \ell}{N^-<p_i<N^+}}}\prod_{j=1}^\ell\frac{1}{p_j(p_j-1)}\sum_{\FFps}h(P,N,S,T) \nonumber\\
+&\frac{2^\ell}{\#\C(A,B)}\sum_{\underset{p_m\neq p_n,\, \forall m\neq n }{\underset{1\le i\le \ell}{N^-<p_i<N^+}}}\prod_{j=1}^\ell\frac{1}{(p_j-1)}\sum_{\FFps}h(P,N,S,T)\left(Z(P,S,T)-
\frac{4AB}{2^\ell p_1\cdots p_\ell}\right).\label{alitwo}
\end{align}
In order to bound the second summation on the right hand side of (\ref{alitwo}), we use the following lemma

\begin{lemma} \label{errorimprove}
Let $\ell, \, A,\, B, \, h(.),\, Z(.) $ as defined before. Then, as $N\rightarrow \infty$, we have
\begin{align}
&\sum_{\underset{p_m\neq p_n,\, \forall m\neq n }{\underset{1\le i\le \ell}{N^-<p_i<N^+}}}\frac{1}{p_1\cdots p_\ell}\sum_{\FFps}h(P,N,S,T) \left(Z(P,S,T)-\frac{AB}{2^{\ell-2}p_1\cdots p_\ell}\right)\nonumber \\
\ll_{k,\ell} &ABN^{-\frac{\ell}{4k}}(\log N)^\frac{\ell}{2k}(\log\log N)^\ell\left((\log A)^{\frac{k^2-1}{2k}}+(\log B)^{\frac{k^2-1}{2k}}\right)\nonumber\\
+&(A\sqrt{B}+B\sqrt{A})N^{\frac{3\ell}{4k}}(\log N)^{\frac{k^2+\ell-1}{2k}}(\log\log N)^\ell+
\sqrt{AB}N^{\frac{3\ell}{4}}(\log N)^{3-\ell}(\log \log N)^\frac{\ell}{2},\label{lemerrorbound}
\end{align}
for any positive integer $k$.
\end{lemma}
We give a proof of the above lemma later in this section. \\

Now, using (\ref{eq_001}), we write the inner summation in the first sum in $\eqref{alitwo}$ as
\begin{align}
\sum_{\FFps}h(P,N,S,T)&=\sum_{\substack{1\leq s_1,t_1<p_1 \\ \#E_{p_1,s_1,t_1}(\FF_{p_1})=N}}\cdots\sum_{\substack{1\leq s_\ell,t_\ell<p_\ell \\
\#E_{p_\ell,s_\ell,t_\ell}(\FF_{p_\ell})=N}}1. \nonumber\\
&=\prod_{i=1}^{\ell}\Big(\sum_{\substack{\bar{E}_{p_i,s_i,t_i}/\FF_{p_i} \\ p_i+1-a_{p_i}(\bar{E}_{p_i,s_i,t_i})=N}}\frac{p_i-1}
{\#\aut(\bar{E}_{p_i,s_i,t_i}(\FF_{p_i}))}+O(p_i)\Big) \nonumber \\
&=\prod_{i=1}^{\ell}\Big((p_i-1)H(D_N(p_i))+O(p_i)\Big). \label{alithree}
\end{align}
Using the bound $L(1,\chi_{d_{N,f}(p)})\ll \log N $, one can show that $H(D_N(p_i))\ll \sqrt{N}\log N\log \log N$ for $1\le i\le \ell$. This together with (\ref{alithree}) gives

\begin{align}
\sum_{\FFps}h(P,N,S,T)=&\prod_{i=1}^{\ell}(p_i-1)H(D_N(p_i))+O_\ell\left(N^{\frac{3\ell-1}{2}}(\log N)^{\ell-1}(\log\log N)^{\ell-1}\right).\label{dubsum}
\end{align}
Using $\eqref{dubsum}$, the first term in $\eqref{alitwo}$ can be written as
\begin{align}
&\frac{4AB}{\#\C(A,B)}\sum_{\underset{p_i\neq p_j,\, \forall i\neq j }{\underset{1\le i\le \ell}{N^-<p_i<N^+}}}\prod_{j=1}^\ell\frac{1}{p_j(p_j-1)}\sum_{\FFps}h(P,N,S,T)\nonumber\\
=&\frac{4AB}{\#\C(A,B)}\sum_{\underset{p_i\neq p_j,\, \forall i\neq j }{\underset{1\le i\le \ell}{N^-<p_i<N^+}}}\Bigg(\prod_{j=1}^{\ell}\frac{H(D_N(p_j))}{p_j}+O_\ell\left(\frac{1}{N^{2\ell}}
\cdot N^{\frac{3\ell-1}{2}}(\log N)^{\ell-1}(\log\log N)^{\ell-1}\right)\Bigg) \nonumber\\
=&\Bigg(\sumofallps\prod_{j=1}^\ell\frac{H(D_N(p_j))}{p_j}+O_\ell\left(\frac{(\log\log N)^{\ell-1}}{\sqrt{N}\log N}\right)\Bigg)\left(1+O_\ell\left(\frac{1}{A}+\frac{1}{B}+\frac{1}{AB}\right)\right).\label{alifive}
\end{align}

Combining (\ref{alifive}) and Lemma \ref{lm_1.1}, together with Lemma \ref{errorimprove}, we can write (\ref{alitwo}) as
\begin{align}
\Sigma_1&=\Bigg(\sum_{N^-<p<N^+}\frac{H(D_N(p))}{p}\Bigg)^{\ell}+\mathcal{E}_1(N,A,B,\ell)\nonumber
\intertext{where}
\mathcal{E}_1(N,A,B,\ell)
&\ll_{\ell,k}N^{-\frac{\ell}{4k}}(\log N)^\frac{\ell}{2k}(\log\log N)^\ell\left((\log A)^{\frac{k^2-1}{2k}}+(\log B)^{\frac{k^2-1}{2k}}\right) +\frac{1}{\sqrt{AB}} N^{\frac{3\ell}{4}}(\log N)^{3-\ell}(\log \log N)^{\frac{\ell}{2}}\nonumber \\
&+ \left(\frac{1}{\sqrt{B}}+\frac{1}{\sqrt{A}}\right)N^{+\frac{3\ell}{4k}}(\log N)^{\frac{k^2+\ell-1}{2k}}(\log\log N)^\ell+
\left(\frac{\log \log N}{\log N}\right)^{\ell}\left(\frac{1}{A}+ \frac{1}{B}+\frac{1}{AB}\right)+O_\ell(N^{-\frac{1}{2}}).\label{eq_002}
\end{align}

For the time being, we assume that $\Sigma_2$ is significantly small compared to $\mathcal{E}_1(N,A,B,\ell)$ under the condition that $A,\, B\ge N^\epsilon$.\\
Choose $k=\frac{2\ell}{\epsilon}$. Now, if
\begin{align*}
N^{\epsilon}\le A,\, B&\le \exp{(N^{\frac{\epsilon^2}{20\ell}})}\\
AB&\ge N^{\frac{3\ell}{2}}(\log N)^{6+2\gamma_2}(\log \log N)^{\frac{\ell}{2}}\\
\frac{\log \log N}{\log N}&\ge \frac{20\ell}{\epsilon^2}
\end{align*}
one can check that $$\mathcal{E}_1(N,A,B,\ell)\ll O\Big(\frac{1}{(\log N)^{\ell+\gamma_2}}\Big).$$

Before we proceed with estimating $\Sigma_2$ as defined in (\ref{eq_003}), we give a proof of Lemma \ref{errorimprove}. Later we are going to use the same proof and the discussions above to give a bound on $\Sigma_2$.

\subsection{Proof of Lemma \ref{errorimprove}:}

Let $\chi_i$ and $\chi_i'$ be Dirichlet characters modulo $p_i$ for $1\leq i\leq \ell$ and let $\chi_0$ denote the principal character modulo $n$ for any integer $n$. Let $${\mathcal{A}}(\chi):=\sum_{|a|\leq A}\chi(a) \quad {\rm and}\quad {\mathcal{B}}(\chi):=\sum_{|b|\leq B}\chi(b).$$

For $U:=(u_1,\ldots, u_\ell)\in \F_{p_1}^*\times \cdots \F_{p_\ell}^*=\F(P)^*$,
\begin{align}
Z(P,S,T)&=\sum_{\substack{ |a|\leq A, |b|\leq B \\ \exists\; U\in\FF(P)^*\\ a\equiv s_iu_i^4\imod{p_i},
b\equiv t_iu_i^6\imod{p_i}\\ 1\leq i \leq \ell}} 1 \nonumber\\
=\frac{1}{2^{\ell}}&\sum_{\substack{|a|\leq A \\ |b|\leq B}}\sum_{ U\in\FF(P)^*}\prod_{i=1}^\ell
\Bigg(\frac{1}{\varphi(p_i)^2}\sum_{\chi_i\imod{p_i}}\chi_i(s_iu_i^4)
\overline{\chi_i}(a)\sum_{\chi_i'\imod{p_i}}\chi_i'(t_iu_i^6)\overline{\chi_i'}(b)\Bigg)\nonumber\\
=\frac{1}{2^{\ell}}&\prod_{i=1}^\ell\frac{1}{(p_i-1)^2}\sum_{ U\in\FF(P)^*}\sum_{\substack{\chi_i, \chi_i'\imod{p_i} \\ 1\leq i \leq \ell}} \chi_i(s_i)\chi_i'(t_i)
\chi_i(u_i^4)\chi_i'(u_i^6)
\sum_{\substack{|a|\leq A \\ |b|\leq B}}\overline{\chi_1\cdots\chi_\ell}(a)\overline{\chi_1'\cdots\chi_\ell'}(b).\label{suv}
\end{align}

By the orthogonality relations of Dirichlet characters, we have
\begin{equation}
\prod_{i=1}^\ell\sum_{ U\in\FF_{p_i}^*}\chi_i(u_i^4)\chi_i'(u_i^6)=\begin{cases}\prod_{i=1}^\ell (p_i-1) & \:{\rm if}\: \chi_i^4(\chi_i')^6=\chi_{0}\imod{p_i}\:{\rm for}\:1\leq i\leq \ell, \\ 0 & \:{\rm otherwise}.\end{cases}
\label{suvtwo}
\end{equation}
Then combining $\eqref{suv}$ and $\eqref{suvtwo}$, we get

\begin{align}
Z(P,S,T)=&\frac{1}{2^\ell}\sum_{\substack{\chi_1,\ldots,\chi_\ell \\ \chi_1',\ldots, \chi_\ell' \\ \chi_i^4(\chi_i')^6= \chi_{0} \imod{p_i} \\
{\rm for}\: 1\leq i\leq \ell}}\prod_{i=1}^\ell\left(\frac{\chi_i(s_i)\chi_i'(t_i)}{p_i-1}\right){\mathcal{A}}(\overline{\chi_1\cdots \chi_\ell})
{\mathcal{B}}(\overline{\chi_1'\cdots \chi_\ell'}) \nonumber\\
=&\frac{1}{2^\ell}\Bigg[\sum_{\substack{\chi_i=\chi_i'=\chi_0 \imod{p_i}\\
{\rm for}\: 1\leq i\leq \ell}} + \sum_{\substack{\chi_i=(\chi_i')^6=\chi_0 \imod{p_i}\\
{\rm for}\: 1\leq i\leq \ell\:{\rm and} \\ \exists 1\leq j \leq \ell \:{\rm s.t.}\: \chi_j'\neq \chi_0 \imod{p_j}}}+\sum_{\substack{\chi_i'=\chi_i^4=\chi_0 \imod{p_i}\\
{\rm for}\: 1\leq i\leq \ell\:{\rm and} \\ \exists 1\leq j \leq \ell \:{\rm s.t.}\: \chi_j\neq \chi_0 \imod{p_j}}}\nonumber \\
+&\sum_{\substack{\chi_i^4(\chi_i')^6=\chi_0 \imod{p_i}\\
{\rm for}\: 1\leq i\leq \ell\:{\rm and} \\ \exists 1\leq r,s \leq \ell \:{\rm s.t.}\: \chi_r\neq \chi_0 \imod{p_r},\\\chi_s'\neq \chi_0 \imod{p_s} }}\Bigg]\prod_{i=1}^\ell\left(\frac{\chi_i(s_i)\chi_i'(t_i)}{p_i-1}\right){\mathcal{A}}(\overline{\chi_1\cdots \chi_\ell})
{\mathcal{B}}(\overline{\chi_1'\cdots \chi_\ell'})\nonumber\\
=& Z_1(P,S,T)+Z_2(P,S,T)+Z_3(P,S,T)+Z_4(P,S,T)\label{thefourcs}
\end{align}

Then, the LHS of $\eqref{lemerrorbound}$, can be written as,
\begin{align}
\sumofpfracsing &\sum_{\FFps}h(P,N,S,T) \left(Z(P,S,T)-\frac{AB}{2^{\ell-2}p_1\cdots p_\ell}\right)\nonumber\\
&=\sumofpfracsing\sum_{\FFps}h(P,N,S,T) \left(\sum_{j=1}^4Z_j(P,S,T)-\frac{AB}{2^{\ell-2}p_1\cdots p_\ell}\right).\label{eq_005}
\end{align}

\emph{Case 1:}
 $\chi_i=\chi_i'=\chi_0 \imod{p_i}$ for $1\leq i \leq \ell$.\\
In this case,
\begin{align}
{\mathcal{A}}(\overline{\chi_1\cdots \chi_\ell})&=\sum_{|a|\leq A}\chi_0(a)=\sum_{\substack {|a|\leq A \\ (a,p_1\cdots p_\ell)=1}}1=2A\frac{\varphi(p_1\cdots p_\ell)}{p_1\cdots p_\ell}+O(\tau(p_1\cdots p_\ell))\nonumber \\
&=2A\left(\frac{(p_1-1)\cdots(p_\ell-1)}{p_1\cdots p_\ell}\right)+O_\ell(1)\label{achistuff}
\end{align}
and
$${\mathcal{B}}(\overline{\chi_1'\cdots \chi_\ell'})=2B\left(\frac{(p_1-1)\cdots(p_\ell-1)}{p_1\cdots p_\ell}\right)+O_\ell(1).$$
Consequently,
\begin{align}
Z_1(P,S,T)=&\frac{1}{2^\ell}\prod_{j=1}^\ell \frac{1}{p_j-1}\left(\frac{2A(p_1-1)\cdots(p_\ell-1)}{p_1\cdots p_\ell}+O_\ell(1)\right)
\left(\frac{2B(p_1-1)\cdots(p_\ell-1)}{p_1\cdots p_\ell}+O_\ell(1)\right)\nonumber\\
=&\frac{AB}{2^{\ell-2}p_1\cdots p_\ell}+O_\ell\left(\frac{AB}{N^{\ell+1}}+\frac{A+B+1}{N^\ell}\right). \label{sjequalsone}
\end{align}

Combining (\ref{dubsum}) with (\ref{sjequalsone}) and using Lemma \ref{lm_1.1}, we get

\begin{align}
&\sumofpfracsing\sum_{\FFps}h(P,S,T) \left(Z_1(P,S,T)-\frac{AB}{2^{\ell-2}p_1\cdots p_\ell}\right)\nonumber \\
\ll_\ell&\sumofpfracsing\left(\frac{AB}{N^{\ell+1}}+\frac{A+B+1}{N^\ell}\right)\left(\prod_{j=1}^{\ell}(p_j-1)H(D_N(p_j))+N^{\frac{3\ell-1}{2}}(\log N)^{\ell-1}(\log \log N)^{\ell-1}\right)
\nonumber\\
\ll_\ell & \frac{AB (\log \log N)}{N(\log N)^{\ell}}+\frac{(A+B+1)}{(\log N)^\ell}.\label{jonebound}
\end{align}

Now,
\begin{align*}
Z_2(P,S,T)&=\frac{1}{2^\ell}\sum_{\substack{\chi_i=(\chi_i')^6=\chi_0 \imod{p_i}\\
{\rm for}\: 1\leq i\leq \ell\:{\rm and} \\ \exists 1\leq j \leq \ell \:{\rm s.t.}\: \chi_j'\neq \chi_0 \imod{p_j}}}\prod_{i=1}^\ell\left(\frac{\chi_i(s_i)\chi_i'(t_i)}{p_i-1}\right){\mathcal{A}}(\overline{\chi_1\cdots \chi_\ell})
{\mathcal{B}}(\overline{\chi_1'\cdots \chi_\ell'})\\
&=\frac{1}{2^\ell} \sum_{\substack{(\chi_i')^6=\chi_0 \imod{p_i}\\
{\rm for}\: 1\leq i\leq \ell\:{\rm and} \\ \exists 1\leq j \leq \ell \:{\rm s.t.}\: \chi_j'\neq \chi_0 \imod{p_j}}} \prod_{j=1}^\ell\frac{\chi_j'(t_j)}{(p_j-1)}\left(2A\prod_{i=1}^\ell\frac{(p_i-1)}{p_i}+O_\ell(1)\right)
{\mathcal{B}}(\overline{\chi_1'\cdots \chi_\ell'})\\
&\ll_\ell\frac{A}{p_1\cdots p_\ell}\sum_{\substack{(\chi_i')^6=\chi_0 \imod{p_i}\\
{\rm for}\: 1\leq i\leq \ell\:{\rm and} \\ \exists 1\leq j \leq \ell \:{\rm s.t.}\: \chi_j'\neq \chi_0 \imod{p_j}}}|{\mathcal{B}}(\overline{\chi_1'\cdots \chi_\ell'})|.
\end{align*}

 Using (\ref{alithree}) and H{\"o}lder's inequality, we get
\begin{align}
\sumofpfracsing &\sum_{\FFps}h(P,N,S,T)(Z_2(P,S,T)\nonumber\\
&  \ll_\ell A\sumofallps\prod_{j=1}^\ell\frac{H(D_N(p_j))}{p_j}\sum_{\substack{(\chi_i')^6=\chi_0 \imod{p_i}\\
{\rm for}\: 1\leq i\leq \ell\:{\rm and} \\ \exists 1\leq j \leq \ell \:{\rm s.t.}\: \chi_j'\neq \chi_0 \imod{p_j}}}  |{\mathcal{B}}(\overline{\chi_1'\cdots \chi_\ell'})|\nonumber\\
&\ll_\ell A\Bigg(\sumpchiprim\prod_{j=1}^\ell\left(\frac{H(D_N(p_j))}{p_j}\right)^{\frac{2k}{2k-1}}\Bigg)^{1-\frac{1}{2k}}\nonumber\\
&\quad \quad \quad \times \Bigg(\sumpchiprim|{\mathcal{B}}(\overline{\chi_1'\cdots \chi_\ell'})|^{2k}\Bigg)^\frac{1}{2k}\nonumber\\
&\ll_\ell A\Bigg(\sum_{\substack{N^-<p_i<N^+ \\ p_i\neq p_j ,\, \forall i\neq j \\ 1\leq i \leq \ell}}\left(\frac{(\log p_i)^\ell (\log \log p_i)^\ell}{p_i^\frac{\ell}{2}}\right)^{\frac{2k}{2k-1}}\Bigg)^{1-\frac{1}{2k}}\nonumber\\
&\quad \quad \quad \times \Bigg(\sumpchiprim|{\mathcal{B}}(\overline{\chi_1'\cdots \chi_\ell'})|^{2k}\Bigg)^\frac{1}{2k}\nonumber\\
 & \ll_\ell  N^{-\frac{\ell}{4k}}(\log N)^{\frac{\ell}{2k}}(\log \log N)^\ell \Bigg(\sumpchiprim|{\mathcal{B}}(\overline{\chi_1'\cdots \chi_\ell'})|^{2k}\Bigg)^\frac{1}{2k}\label{holdboundone}
\end{align}

Now, for a fixed prime $\ell-tuple$ $(p_1,p_2,\cdots, p_\ell)$, the second product in (\ref{holdboundone}),
let $J\subseteq\{1,\ldots,\ell\}$ be the set of positive integers such that $\chi_j'\neq \chi_0 \imod {p_j}$ for  $j\in J$. Thus, $$|{\mathcal{B}}(\overline{\chi_1'\cdots \chi_\ell'})|=\Bigg|\sum_{|b|\leq B}\overline{\chi_1'}(b)\cdots\overline{\chi_\ell'}(b) \Bigg|=\Bigg|\sum_{|b|\leq B}\prod_{j\in J}\overline{\chi_j'}(b)\prod_{j\not \in J}\overline{\chi_j'}(b)\Bigg|=\Bigg|\sum_{\substack{|b|\leq B \\ (b,\prod_{j\not \in J} p_j)=1}}\prod_{j\in J}\overline{\chi_j'}(b)\Bigg|.$$ Let $\tau_k(b;B)$ denote the number of representation of $b$ as a product of $k$ positive $B$ smooth integers. Then,
\begin{align*}\Bigg|\sum_{\substack{|b|\leq B \\ (b,\prod_{j\not \in J} p_j)=1}}\prod_{j\in J}\overline{\chi_j'}(b)\Bigg|^{2k}\ll_\ell \Bigg|\sum_{\substack{b\leq B^k \\ (b,\prod_{j\not\in J} p_j)=1}}\tau_k(b;B)\prod_{j\in J}\overline{\chi_j'}(b)\Bigg|^2.\end{align*}
Thus,
\begin{align}
\Bigg(\sumpchiprim|{\mathcal{B}}(\overline{\chi_1'\cdots \chi_\ell'})|^{2k}&\Bigg)^\frac{1}{2k}\nonumber\\ \ll_\ell \Bigg(\sumpchiprim &\Bigg|\sum_{\substack{b\leq B^k \\ (b,\prod_{j\not\in J} p_j)=1}}\tau_k(b;B)\prod_{j\in J}\overline{\chi_j'}(b)\Bigg|^2\Bigg)^\frac{1}{2k}.\label{holdboundtwo}
\end{align}

Now, $\Big(\overline{\prod_{j\in J}{\chi_j'}}\Big)(b)$ is a primitive character modulo $\prod_{j\in J} p_j\le N^\ell$. Now we extend the sum in $\eqref{holdboundtwo}$ to a sum over all primitive characters modulo $d$ for all
modulus $d\leq N^{\ell}$.
Using Theorem \ref{largesieve}, we get
\begin{align}
\Bigg(\sumpchiprim|{\mathcal{B}}(\overline{\chi_1'\cdots \chi_\ell'})|^{2k}\Bigg)^\frac{1}{2k}
&\ll_\ell \Bigg(\sum_{\substack{d\leq N^\ell \\\chi\imod{d}\\ \chi \: {\rm primitive}}}\Bigg|\sum_{b\leq B^k}\tau_k(b;B)\chi(b)\Bigg|^2\Bigg)^\frac{1}{2k}\nonumber\\
&\ll_\ell \Bigg(\sum_{\substack{d\leq N^\ell \\\chi\imod{d}\\ \chi \: {\rm primitive}}}\Bigg|\sum_{b\leq B^k}\tau_k(b)\chi(b)\Bigg|^2\Bigg)^\frac{1}{2k}\nonumber\\
&\ll_\ell \Bigg((B^k+N^{2\ell})\sum_{b\leq B^k}\left|\tau_k(b)\right|^2\Bigg)^{\frac{1}{2k}}\nonumber\\
& \ll_\ell \left((B^k+N^{2\ell})B^k\log^{k^2-1}(B^k)\right)^{\frac{1}{2k}}\label{lalsieveboundprim}
\end{align}

Combining $ \eqref{holdboundone}$ and $\eqref{lalsieveboundprim}$, we get
\begin{align}
&\sumofpfracsing\sum_{\FFps}h(P,N,S,T)Z_2(P,S,T)\nonumber\\
&\quad \quad \ll_\ell A\sumofpfracsing\prod_{j=1}^{\ell}H(D_N(p_j))\sum_{\substack{(\chi_i')^6=\chi_0 \imod{p_i}\\
{\rm for}\: 1\leq i\leq \ell\:{\rm and} \\ \exists 1\leq j \leq \ell \:{\rm s.t.}\: \chi_j'\neq \chi_0 \imod{p_j}}} |{\mathcal{B}}(\overline{\chi_1'\cdots \chi_\ell'})|\nonumber\\
&\quad \quad \ll_\ell A\left((B^k+N^{2\ell})B^k\log^{k^2-1}(B^k)\right)^{\frac{1}{2k}}N^{-\frac{\ell}{4k}}(\log N)^{\frac{\ell}{2k}}(\log \log N)^\ell
\nonumber\\
&\quad \quad \ll_{\ell,k}ABN^{-\frac{\ell}{4k}}(\log N)^{\frac{\ell}{2k}}(\log\log N)^\ell\log^{\frac{k^2-1}{2k}}B+A\sqrt{B}N^{\frac{3\ell}{4k}}(\log N)^{\frac{k^2+\ell-1}{2k}}(\log\log N)^\ell \label{secondtermbound}.
\end{align}

Following almost similar arguments,

$$Z_3(P,S,T)\ll_\ell\frac{B}{p_1\cdots p_\ell}\sum_{\substack{\chi_i^4=\chi_0 \imod{p_i}\\
{\rm for}\: 1\leq i\leq \ell\:{\rm and} \\ \exists 1\leq j \leq \ell \:{\rm s.t.}\: \chi_j\neq \chi_0 \imod{p_j}}}|{\mathcal{A}}(\overline{\chi_1\cdots \chi_\ell})|.$$
and
\begin{align}
\sumofpfracsing &\sum_{\FFps}h(P,N,S,T)Z_3(P,S,T)\nonumber\\
& \ll B\sumofpfracsing\prod_{j=1}^{\ell}H(D_N(p_j))\sum_{\substack{\chi_i^4=\chi_0 \imod{p_i}\\
{\rm for}\: 1\leq i\leq \ell\:{\rm and} \\ \exists 1\leq j \leq \ell \:{\rm s.t.}\: \chi_j\neq \chi_0 \imod{p_j}}} |{\mathcal{A}}(\overline{\chi_1\cdots \chi_\ell})|\nonumber \\
& \ll_{\ell,k} ABN^{-\frac{\ell}{4k}}(\log N)^{\frac{\ell}{2k}}(\log\log N)^\ell\log^{\frac{k^2-1}{2k}}A+B\sqrt{A}N^{\frac{3\ell}{4k}}(\log N)^{\frac{k^2+\ell-1}{2k}}(\log\log N)^\ell \label{sjthreebound}
\end{align}
for a positive real number $k>\frac{1}{2}$.
Hence,
\begin{align}
&\sumofpfracsing\sum_{\FFps}h(P,N,S,T)(Z_2(P,S,T)+Z_3(P,S,T))\nonumber \\
&\ll_{k,\ell} ABN^{-\frac{\ell}{4k}}(\log N)^{\frac{\ell}{2k}}(\log\log N)^\ell(\log^{\frac{k^2-1}{2k}}A+\log^{\frac{k^2-1}{2k}}B)
+(A\sqrt{B}+B\sqrt{A})N^{\frac{3\ell}{4k}}(\log N)^{\frac{k^2+\ell}{2k}}(\log\log N)^\ell \label{jtwosevenbound}
\end{align}

Finally, for  $Z_4(P,S,T)$, define $$g(P,\chi_i,\chi_i'):= \sum_{\substack{1\leq s_i,t_i<p_i \\ 1\leq i\leq \ell}}h(P,N,S,T)\chi_i(s_i)\chi_i'(t_i).$$ Then,
\begin{align}
&\sumofpfracsing\sum_{\FFps}h(P,N,S,T)Z_4(P,S,T)\nonumber\\
=&\frac{1}{2^\ell}\sumofallps \prod_{j=1}^\ell \frac{1}{p_j(p_j-1)}\sum_{\substack{\chi_i^4(\chi_i')^6=\chi_0 \imod{p_i}\\
{\rm for}\: 1\leq i\leq \ell\:{\rm and} \\ \exists 1\leq r,s \leq \ell \:{\rm s.t.}\: \chi_r\neq \chi_0 \imod{p_r},\\\chi_s'\neq \chi_0 \imod{p_s} }}g(P,\chi_i,\chi_i'){\mathcal{A}}(\overline{\chi_1\cdots \chi_\ell}){\mathcal{B}}(\overline{\chi_1'\cdots \chi_\ell'}).\label{sfoursumblah}
\end{align}
Applying H{\"o}lder's inequality again, we have
\begin{align}
&\Bigg|\sum_{\substack{\chi_i^4(\chi_i')^6=\chi_0 \imod{p_i}\\
{\rm for}\: 1\leq i\leq \ell\:{\rm and} \\ \exists 1\leq r,s \leq \ell \:{\rm s.t.}\: \chi_r\neq \chi_0 \imod{p_r},\\\chi_s'\neq \chi_0 \imod{p_s} }}g(P,\chi_i,\chi_i'){\mathcal{A}}(\overline{\chi_1\cdots \chi_\ell}){\mathcal{B}}(\overline{\chi_1'\cdots \chi_\ell'})\Bigg| \nonumber\\
\leq&\Bigg|\sum_{\substack{\chi_i^4(\chi_i')^6=\chi_0 \imod{p_i}\\
{\rm for}\: 1\leq i\leq \ell \:{\rm and} \\ \exists 1\leq r,s \leq \ell \:{\rm s.t.}\: \chi_r\neq \chi_0 \imod{p_r},\\\chi_s'\neq \chi_0 \imod{p_s} }}\left|g(P,\chi_i,\chi_i')\right|^2\Bigg|^\frac{1}{2}\Bigg(\sum_{\substack{\chi_i^4(\chi_i')^6=\chi_0 \imod{p_i}\\
{\rm for}\: 1\leq i\leq \ell\:{\rm and} \\ \exists 1\leq r,s \leq \ell \:{\rm s.t.}\: \chi_r\neq \chi_0 \imod{p_r},\\\chi_s'\neq \chi_0 \imod{p_s} }}\left|{\mathcal{A}}(\overline{\chi_1\cdots \chi_\ell})\right|^4\Bigg)^{\frac{1}{4}} \nonumber\\
\times&\Bigg(\sum_{\substack{\chi_i^4(\chi_i')^6=\chi_0 \imod{p_i}\\
{\rm for}\: 1\leq i\leq \ell\:{\rm and} \\ \exists 1\leq r,s \leq \ell \:{\rm s.t.}\: \chi_r\neq \chi_0 \imod{p_r},\\\chi_s'\neq \chi_0 \imod{p_s} }}\left|{\mathcal{B}}(\overline{\chi_1'\cdots \chi_\ell'})\right|^4\Bigg)^\frac{1}{4}.\label{jfourcauchy}
\end{align}

Now, extending the sum over all non-principal characters modulo $N^\ell$, from Theorem \ref{fourthpower} we have
\begin{align}
\sum_{\substack{\chi_i^4(\chi_i')^6=\chi_0 \imod{p_i}\\
{\rm for}\: 1\leq i\leq \ell\:{\rm and} \\ \exists 1\leq r,s \leq \ell \:{\rm s.t.}\: \chi_r\neq \chi_0 \imod{p_r},\\\chi_s'\neq \chi_0 \imod{p_s} }}\left|{\mathcal{A}}(\overline{\chi_1\cdots \chi_\ell})\right|^4
\ll_\ell& \sum_{\chi\neq \chi_0 \imod{N^\ell}}\Bigg|\sum_{|a|\leq A}\overline{\chi}(a)\Bigg|^4 \nonumber \\
\ll_\ell &{A^2 N^\ell}(\log (N^\ell))^6 \ll_\ell A^2N^\ell(\log N)^6 \label{jfourone}
\end{align}

Similarly,

\begin{align}
\sum_{\substack{\chi_i^4(\chi_i')^6=\chi_0 \imod{p_i}\\
{\rm for}\: 1\leq i\leq \ell\:{\rm and} \\ \exists 1\leq r,s \leq \ell \:{\rm s.t.}\: \chi_r\neq \chi_0 \imod{p_r},\\\chi_s'\neq \chi_0 \imod{p_s} }}\left|{\mathcal{B}}(\overline{\chi_1'\cdots \chi_\ell'})\right|^4
\ll_\ell &\Bigg(\sum_{\chi'\neq \chi_0 \imod{N^\ell}}\Bigg|\sum_{|b|\leq B}\overline{\chi'}(b)\Bigg|^4\nonumber\\
\ll_\ell &{B^2 N^\ell}(\log (N^\ell))^6 \ll_\ell B^2N^\ell(\log N)^6 \label{eq_011}
\end{align}

 Further, from (\ref{jfourcauchy}), we have
\begin{align}
&\sum_{\substack{\chi_i^4(\chi_i')^6=\chi_0 \imod{p_i}\\
{\rm for}\: 1\leq i\leq \ell\:{\rm and} \\ \exists 1\leq r,s \leq \ell \:{\rm s.t.}\: \chi_r\neq \chi_0 \imod{p_r},\\\chi_s'\neq \chi_0 \imod{p_s} }}\left|g(P,\chi_i,\chi_i')\right|^2\leq \sum_{\substack {\chi_i,\chi_i'\imod{p_i}\\ 1\leq i \leq \ell}}\left|g(P,\chi_i,\chi_i')\right|^2\nonumber \\
&\leq \sum_{\FFps}\sum_{S',T'\in\FF(P)^*}
h(P,N,S,T)\overline{h(P,N,S',T')}\sum_{\chi_i\imod{p_i}}\chi_i(s_i)\overline{\chi_i}(s_i')\sum_{\chi_i'\imod{p_i}}\chi_i'(t_i)\overline{\chi_i'}(t_i')\nonumber\\
&=\prod_{i=1}^\ell (p_i-1)^2\sum_{\FFps}\left|h(P,N,S,T)\right|\nonumber\\
&=\prod_{i=1}^\ell (p_i-1)^2\sum_{\FFps}\left|h(P,N,S,T)\right|^2\nonumber\\
&=N^{3\ell}\prod_{i=1}^{\ell}H(D_N(p_i))+O_\ell\left(N^{\frac{7\ell-1}{2}}(\log N)^\ell(\log \log N)^{\ell}\right),\label{jfourtwo}
\end{align}

Combining (\ref{jfourcauchy}), (\ref{jfourone}), (\ref{eq_011}) and (\ref{jfourtwo}), we have
 \begin{align}
&\Bigg|\sum_{\substack{\chi_i^4(\chi_i')^6=\chi_0 \imod{p_i}\\
{\rm for}\: 1\leq i\leq \ell\:{\rm and} \\ \exists 1\leq r,s \leq \ell \:{\rm s.t.}\: \chi_r\neq \chi_0 \imod{p_r},\\\chi_s'\neq \chi_0 \imod{p_s} }}g(P,\chi_i,\chi_i'){\mathcal{A}}(\overline{\chi_1\cdots \chi_\ell}){\mathcal{B}}(\overline{\chi_1'\cdots \chi_\ell'})\Bigg| \nonumber\\
 \ll_\ell &\sqrt{AB}N^{2\ell}(\log N)^3\prod_{i=1}^{\ell}(H(D_N(p_i)))^2\label{wpfourstuff}
 \end{align}

 Thus $\eqref{wpfourstuff}$ and $\eqref{sfoursumblah}$ gives
\begin{align}
&\sumofpfracsing\sum_{\FFps}h(P,N,S,T)Z_4(P,S,T)\nonumber\\
\ll_\ell&\sqrt{AB}(\log N)^3 \sumofallps \prod_{i=1}^{\ell}\sqrt{H(D_N(p_j))} \label{jfourprecauchy}
\end{align}

Using Lemma \ref{lm_1.1} and Cauchy-Schwarz inequality, we get
\begin{align}
\sumofallps \prod_{i=1}^{\ell}\sqrt{H(D_N(p_j))}
\ll_\ell& \Bigg(\sumofallps \prod_{i=1}^{\ell}H(D_N(p_i))\Bigg)^{\frac{1}{2}}\Bigg(\sumofallps 1\Bigg)^{\frac{1}{2}} \nonumber \\
\ll_\ell & \frac{N^{\frac{3\ell}{4}}(\log \log N)^{\frac{\ell}{2}}}{(\log N)^\ell} \label{cauchyhbound}
\end{align}

Thus,
\begin{align}
&\sumofpfracsing\sum_{\FFps}h(P,N,S,T)Z_4(P,S,T)
 \ll_\ell\sqrt{AB}N^{\frac{3\ell}{4}}(\log N)^{3-\ell}(\log \log N)^{\frac{\ell}{2}}
 \label{jeightbound}
  \end{align}
Finally, combining (\ref{jonebound}), (\ref{jtwosevenbound}) and (\ref{jeightbound}), we complete the proof of Lemma \ref{errorimprove}.

\subsection{Bound on $\Sigma_2$:}

Next, we plan to modify the previous proof of Lemma \ref{errorimprove} to give a upper bound on $\Sigma_2$.\\
Recall,
\begin{align}
\Sigma_2&=\frac{1}{\#\C(A,B)} \sumofallps \sum_{\underset{\underset{1\le i\le \ell}{s_it_i=0 \text{ for some }i}}{S,T\in \F(P)}}h(P,N,S,T)Z(P,S,T)\prod_{j=1}^{\ell}\frac{\#Aut(E_{p_j,s_j,t_j})}{p_j-1}\label{e1}
\end{align}
\emph{Case 1: $s_it_i=0$ for all $i$.}\\
Then the corresponding rational curves look like $E_{a,b}$ where $p_1p_2\cdots p_\ell\mid a$ or $p_1p_2\cdots p_\ell\mid b$. \\
In that case, the contribution corresponding to $ab\neq 0$ is bounded by
\begin{align}
\frac{1}{4AB}\sumofallps \frac{AB}{p_1p_2\cdots p_\ell}\ll_\ell  N^{-\frac{\ell}{2}} \label{e2}
\end{align}

If, either $a=0$ or $b=0$, then the curve has complex multiplication. Hence, by Kowalski \cite{8}, there are only $O_{\epsilon, \ell}(N^{\frac{\epsilon}{2\ell}})$ many primes such that $\#E_p(\F_p)=N$. So, the contribution corresponding to $ab=0$ is bounded by
\begin{align}
 O_{\epsilon, \ell}\Big(\frac{N^{\frac{\epsilon}{2}}(A+B)}{AB}\Big)= O_{\epsilon, \ell}\Big(N^{\frac{\epsilon}{2}}\left(\frac{1}{A}+\frac{1}{B}\right)\Big)\label{eq_017}
\end{align}
\emph{Case 2: $s_{j_1}s_{j_1}\neq 0$ for some $j_1$ and $s_{j_2}s_{j_2}= 0$ for some $j_2$.}\\
The number of possible subsets $I$ of $\{1,2,\cdots \ell \}$ such that $s_it_i=0$ for all $i\in I$ is bounded by $O_\ell(1)$. Take one such subset $I$ and without loss of generality, assume that $\#I=e+f$ with $$s_1=s_2\cdots=s_e=0, \quad t_{e+1}=t_{e+2}=\cdots =t_{e+f}=0, \, \text{ and } s_it_i\neq 0  \text{ for }e+f+1\le i\le \ell.$$
In that case, the contribution corresponding to the set $I$ in (\ref{e1}) is bounded by
\begin{align}
 \frac{1}{4AB}\Bigg(\sumofallpsx \prod_{i=1}^{e+f}\frac{1}{(p_i-1)}\Bigg) \sumofallpsy \prod_{i=e+f+1}^\ell \frac{1}{p_i-1}\sum_{\hat{S},\hat{T}\in \F(\hat{P})^*}\hat{h}(\hat{P},N,\hat{S},\hat{T})\hat{Z}(\hat{P},\hat{S},\hat{T})\label{e3}
\end{align}
where \begin{align*}
       \hat{P}:&=(p_{e+f+1},p_{e+f+2},\cdots, p_\ell)\\
\hat{S}:&=(s_{e+f+1},s_{e+f+2},\cdots, s_\ell)\\
\hat{T}:&=(t_{e+f+1},t_{e+f+2}, \cdots, t_\ell)\\
\hat{h}(\hat{P},N,\hat{S},\hat{T}):&=\begin{cases} 1 & {\rm if}\: \#E_{p_i,s_i,t_i}(\FF_{p_i})=N\:{\rm for}\: e+f+1\leq i\leq \ell, \\ 0 & {\rm otherwise}. \end{cases}\\
\end{align*}

\begin{align*}
\hat{Z}(\hat{P},\hat{S},\hat{T}):&=\sum_{\underset{\underset{\text{for some }(u_{e+f+1},\cdots u_\ell) \in \F(\hat{P})^*}{\underset{p_{e+1}\cdots p_{e+f}b\equiv t_iu_i^6\, (\text{mod }p_i)}{p_1\cdots p_ea\equiv s_iu_i^4\, (\text{mod }p_i)}}}{\underset{ |p_{e+1}\cdots p_{e+f}b|\le B}{|p_1\cdots p_ea|\le A,}}} 1\\
&=\frac{1}{2^{\ell}}\sum_{\substack{|a|\leq A/p_1\cdots p_e \\ |b|\leq B/p_{e+1}\cdots p_{e+f}}}\sum_{ \hat{U}\in\FF(\hat{P})^*}\prod_{i=e+f+1}^\ell\\
& \Bigg(\frac{1}{\varphi(p_i)^2} \sum_{\chi_i\imod{p_i}}\chi_i(s_iu_i^4)
\overline{\chi_i}(p_1\cdots p_ea) \sum_{\chi_i'\imod{p_i}}\chi_i'(t_iu_i^6)\overline{\chi_i'}(p_{e+1}\cdots p_{e+f}b|)\Bigg)
      \end{align*}

Then, (\ref{e3}) is bounded by
\begin{align}
 \frac{1}{4AB}\sumofallpsx (\prod_{i=1}^{e+f}\frac{1}{p_i-1})\hat{\mathcal{E}_1}(N,A,B,e+f+1, \ell)=O\Big(\frac{\hat{\mathcal{E}_1}(N,A,B,e+f+1, \ell)}{ABN^{\frac{e+f}{2}}}\Big)\label{e5}
\end{align}
where 
$$\hat{\mathcal{E}_1}(N,A,B,e+f+1, \ell)=\sumofallpsy \prod_{i=e+f+1}^\ell \frac{1}{p_i-1}\sum_{\hat{S},\hat{T}\in \F(\hat{P})^*}\hat{h}(\hat{P},N,\hat{S},\hat{T})\hat{Z}(\hat{P},\hat{S},\hat{T})$$
We proceed with a argument almost similar to the proof of Lemma \ref{errorimprove} to estimate
\begin{align}
\hat{\mathcal{E}_1}(N,A,&B,e+f+1,\ell)=\sumofallpsy (\prod_{i=e+f+1}^\ell\frac{1}{p_i-1})\sum_{\hat{S},\hat{T}\in \F(\hat{P})^*}\hat{h}(\hat{P},N,\hat{S},\hat{T})\hat{Z}(\hat{P},\hat{S},\hat{T})\nonumber\\
&=\sumofallpsy (\prod_{i=e+f+1}^\ell\frac{1}{p_i-1})\sum_{\hat{S},\hat{T}\in \F(\hat{P})^*}\hat{h}(\hat{P},N,\hat{S},\hat{T})\frac{AB/p_1p_2\cdots p_{e+f}}{2^{\ell-e-f-2}p_{e+f+1}\cdots p_\ell}\nonumber\\
&+\sumofallpsy (\prod_{i=e+f+1}^\ell\frac{1}{p_i-1})\sum_{\hat{S},\hat{T}\in \F(\hat{P})^*}\hat{h}(\hat{P},N,\hat{S},\hat{T})\Bigg(\hat{Z}(\hat{P},\hat{S},\hat{T})-\frac{AB/p_1p_2\cdots p_{e+f}}{2^{\ell-e-f-2}p_{e+f+1}\cdots p_{\ell}}\Bigg)\label{e4}
\end{align}
  Define  $\hat{\mathcal{A}}(\overline{\chi_{e+f+1}\cdots \chi_{\ell}})$ and  $\hat{\mathcal{B}}(\overline{\chi'_{e+f+1}\cdots \chi'_\ell})$, by

\begin{align*}
 \hat{\mathcal{A}}(\overline{\chi_{e+f+1}\cdots \chi_\ell})&=\sum_{|a|\le {A}/{p_1\cdots p_e}}\overline{\chi_{e+f+1}\cdots \chi_\ell}(p_1,p_2\cdots p_e a)\\
&=\overline{\chi_{e+f+1}\cdots \chi_\ell}(p_1,\cdots p_e)\sum_{|a|\le {A}/{p_1\cdots p_e}}\overline{\chi_{e+f+1}\cdots \chi_\ell}(a)\\
\hat{\mathcal{B}}(\overline{\chi'_{e+f+1}\cdots \chi'_\ell})&=\sum_{|b|\le {B}/{p_{e+1}\cdots p_{e+f}}}\overline{\chi'_{e+f+1}\cdots \chi'_\ell}(p_{e+1},\cdots p_{e+f} b)\\
&=\overline{\chi'_{e+f+1}\cdots \chi'_\ell}(p_{e+1},\cdots p_{e+f})\sum_{|b|\le {B}/{p_{e+1}\cdots p_{e+f}}}\overline{\chi'_{e+f+1}\cdots \chi'_\ell}(b)
\end{align*}

First of all, using (\ref{alithree}), note that the first summation on the right hand side of (\ref{e4}) is bounded by \begin{align*}O_\ell\Big(ABN^{-e-f}\Big(\sum_{N-<p<N^+}\frac{H(D_N(p))}{p}\Big)^{\ell-e-f}\Big)=O_\ell\Big(\frac{AB}{N^{e+f}}\Big(\frac{\log \log N}{\log N}\Big)^{\ell-e-f}\Big)\end{align*}

Again, we write $\hat{Z}(\hat{P},\hat{S},\hat{T})$ as $$\hat{Z}(\hat{P},\hat{S},\hat{T})=\sum_{j=1}^4\hat{Z}_j(\hat{P},\hat{S},\hat{T})$$
where
\begin{align*}
\hat{Z}_1(\hat{P},\hat{S},\hat{T})&=\frac{1}{2^\ell}\sum_{\substack{\chi_i=\chi_i'=\chi_0 \imod{p_i}\\ {\rm for}\: 1\leq i\leq \ell}}]\prod_{i=e+f+1}^\ell\left(\frac{\chi_i(s_i)\chi_i'(t_i)}{p_i-1}\right)\hat{{\mathcal{A}}}(\overline{\chi_{e+f+1}\cdots \chi_\ell})
{\hat{\mathcal{B}}}(\overline{\chi_{e+f+1}'\cdots \chi_\ell'})\\
\hat{Z}_2(\hat{P},\hat{S},\hat{T})&=\frac{1}{2^\ell}\sum_{\substack{\chi_i=(\chi_i')^6=\chi_0 \imod{p_i}\\
{\rm for}\: 1\leq i\leq \ell\:{\rm and} \\ \exists 1\leq j \leq \ell \:{\rm s.t.}\: \chi_j'\neq \chi_0 \imod{p_j}}}\prod_{i=e+f+1}^\ell\left(\frac{\chi_i(s_i)\chi_i'(t_i)}{p_i-1}\right)\hat{{\mathcal{A}}}(\overline{\chi_{e+f+1}\cdots \chi_\ell})
{\hat{\mathcal{B}}}(\overline{\chi_{e+f+1}'\cdots \chi_\ell'})\\
\hat{Z}_3(\hat{P},\hat{S},\hat{T})&=\frac{1}{2^\ell}\sum_{\substack{\chi_i'=\chi_i^4=\chi_0 \imod{p_i}\\
{\rm for}\: 1\leq i\leq \ell\:{\rm and} \\ \exists 1\leq j \leq \ell \:{\rm s.t.}\: \chi_j\neq \chi_0 \imod{p_j}}}\prod_{i=e+f+1}^\ell\left(\frac{\chi_i(s_i)\chi_i'(t_i)}{p_i-1}\right)\hat{{\mathcal{A}}}(\overline{\chi_{e+f+1}\cdots \chi_\ell})
{\hat{\mathcal{B}}}(\overline{\chi_{e+f+1}'\cdots \chi_\ell'})\\
\hat{Z}_4(\hat{P},\hat{S},\hat{T})&=\frac{1}{2^\ell}\sum_{\substack{\chi_i^4(\chi_i')^6=\chi_0 \imod{p_i}\\
{\rm for}\: 1\leq i\leq \ell\:{\rm and} \\ \exists 1\leq r,s \leq \ell \:{\rm s.t.}\: \chi_r\neq \chi_0 \imod{p_r},\\\chi_s'\neq \chi_0 \imod{p_s} }}\prod_{i=e+f+1}^\ell\left(\frac{\chi_i(s_i)\chi_i'(t_i)}{p_i-1}\right)\hat{{\mathcal{A}}}(\overline{\chi_{e+f+1}\cdots \chi_\ell})
{\hat{\mathcal{B}}}(\overline{\chi_{e+f+1}'\cdots \chi_\ell'})
\end{align*}

Now, let us denote $\hat{A}=\frac{A}{p_1\cdots p_e}$ and $\hat{B}=\frac{B}{p_{e+1}\cdots p_{e+f}}$.

Then, following the same argument as the one we used to prove (\ref{jonebound}), we should get
\begin{align}
\sumofallpsy & \Big(\prod_{i=e+f+1}^\ell\frac{1}{p_i-1}\Big)\sum_{\hat{S},\hat{T}\in \F(\hat{P})^*}\hat{h}(\hat{P},N,\hat{S},\hat{T})\Bigg(\hat{Z}_1(\hat{P},\hat{S},\hat{T})-\frac{AB/p_1p_2\cdots p_{e+f}}{2^{\ell-e-f-2}p_{e+f+1}\cdots p_\ell}\Bigg)\nonumber\\
&\ll_{\ell} \frac{\hat{A}\hat{B} (\log \log N)}{N(\log N)^{\ell-e-f}}+\frac{(\hat{A}+\hat{B}+1)}{(\log N)^{\ell-e-f}} \label{eq_014}
\end{align}

Since, the primes $p_i$'s are distinct, we also have
\begin{align}
 |\overline{\chi_{m+n+1}\cdots \chi_\ell}(p_1,\cdots p_m)|&=1 \nonumber\\
|\overline{\chi'_{m+n+1}\cdots \chi'_\ell}(p_{m+1},\cdots p_{m+n})|&=1,\label{eq_013}
\end{align}

Hence,
\begin{align*}
\hat{Z}_2(\hat{P},\hat{S},\hat{T})&\ll_\ell\frac{\hat{A}}{p_{e+f+1}\cdots p_\ell}\sum_{\substack{(\chi_i')^6=\chi_0 \imod{p_i}\\
{\rm for}\: e+f+1\leq i\leq \ell\:{\rm and} \\ \exists\, e+f+1\leq j \leq \ell \:{\rm s.t.}\: \chi_j'\neq \chi_0 \imod{p_j}}}|{\mathcal{\hat{B}}}(\overline{\chi_{e+f+1}'\cdots \chi_\ell'})|\\
&\ll_\ell \frac{A}{N^{\ell}}\sum_{\substack{(\chi_i')^6=\chi_0 \imod{p_i}\\
{\rm for}\: e+f+1\leq i\leq \ell\:{\rm and} \\ \exists\, e+f+1\leq j \leq \ell \:{\rm s.t.}\: \chi_j'\neq \chi_0 \imod{p_j}}}\Big|\sum_{|b|\le {B}/{p_{e+1}\cdots p_{e+f}}}\overline{\chi'_{e+f+1}\cdots \chi'_\ell}(b)\Big|
\intertext{and}
\hat{Z}_3(\hat{P},\hat{S},\hat{T})&\ll_\ell \frac{B}{N^{\ell}}\sum_{\substack{\chi_i^4=\chi_0 \imod{p_i}\\
{\rm for}\: e+f+1\leq i\leq \ell\:{\rm and} \\ \exists \, e+f+1\leq j \leq \ell \:{\rm s.t.}\: \chi_j\neq \chi_0 \imod{p_j}}}\Big|\sum_{|a|\le {A}/{p_1\cdots p_e}}\overline{\chi_{e+f+1}\cdots \chi_\ell}(a)\Big|
\end{align*}
Replacing $A$, $B$ and $\ell$ by $\hat{A}$, $\hat{B}$ and $\ell-e-f$ respectively in the proof of (\ref{jtwosevenbound}), we get the following inequality
\begin{align}
  \sumofallpsy & \Big(\prod_{i=e+f+1}^\ell\frac{1}{p_i-1}\Big)\sum_{\hat{S},\hat{T}\in \F(\hat{P})^*}\hat{h}(\hat{P},N,\hat{S},\hat{T})\Bigg(\hat{Z}_2(\hat{P},\hat{S},\hat{T})+\hat{Z}_3(\hat{P},\hat{S},\hat{T})\Bigg)\nonumber\\
&\ll_{k,\ell} \hat{A}\hat{B}N^{-\frac{\ell-e-f}{4k}}(\log N)^{\frac{\ell-e-f}{2k}}(\log\log N)^{\ell-e-f}(\log^{\frac{k^2-1}{2k}}\hat{A}+\log^{\frac{k^2-1}{2k}}\hat{B})\nonumber\\
&\quad +(\hat{A}\sqrt{\hat{B}}+\hat{B}\sqrt{\hat{A}})N^{\frac{3(\ell-e-f)}{4k}}(\log N)^{\frac{k^2+\ell-e-f}{2k}}(\log\log N)^{\ell-e-f}  \label{eq_015}
\end{align}
for some $k>\frac{1}{2}$.

Again, using (\ref{eq_013}) and replacing $A$, $B$ and $\ell$  by $\hat{A}$, $\hat{B}$ and $\ell-e-f$ respectively in the proof of (\ref{jeightbound}), we get
\begin{align}
\sumofallpsy  \Big(\prod_{i=e+f+1}^\ell\frac{1}{p_i-1}\Big)&\sum_{\hat{S},\hat{T}\in \F(\hat{P})^*}\hat{h}(\hat{P},N,\hat{S},\hat{T})\hat{Z}_4(\hat{P},\hat{S},\hat{T})\nonumber\\
&\ll \sqrt{\hat{A}\hat{B}}N^{\frac{3(\ell-e-f)}{4}}(\log N)^{3-(\ell-e-f)}(\log \log N)^{\frac{\ell-e-f}{2}}\label{eq_016}
\end{align}

Since $e+f\ge 1$, observe that we get a savings of a factor of $\frac{1}{\sqrt{N}}$ in $(\ref{eq_014})$, (\ref{eq_015}) and (\ref{eq_016}) compared to the upper bounds for corresponding expressions in the proof of Lemma \ref{errorimprove}. Also, in view of (\ref{eq_017}), we need to assume $A,\, B\ge N^{\epsilon}$ to make the corresponding error term sufficiently small in Proposition \ref{lm_0.1}.

As a conclusion, it is safe to claim that $\Sigma_2$ is small enough compared to the the error term in Proposition \ref{lm_0.1}. This completes the proof of Proposition \ref{lm_0.1}.\qed

\bibliographystyle{amsalpha}

\begin{thebibliography}{PTW02}

\bibitem[AF15]{19} A. Akbary and A. T. Felix, On invariants of elliptic curves on average. Acta Arith. \textbf{168} (2015), no. 1, 31-70.
\medskip
\bibitem[BG14]{10} R. Balasubramanian and S. Giri,  The mean-value of a product of shifted multiplicative functions and the average number of points of elliptic curves. J. Number Theory \textbf{157} (2015), 37-53.
\medskip
\bibitem[BG15]{6} R. Balasubramanian and S. Giri, Poisson Distribution of a Prime Counting Function Corresponding to Elliptic Curves, to appear in Internat. Math. Res. Notices, arXiv:1503.01018 [math.NT].
\medskip
 \bibitem[BCD11]{1} A. Balog, A.-C. Cojocaru and C. David, Average twin prime conjecture for elliptic curves, Amer. J. Math. \textbf{133} (2011), no. 5, 1179-1229.
\medskip
\bibitem[BS09]{18} W. D. Banks and I. E. Shparlinski, Sato-Tate, cyclicity, and divisibility statistics on average for elliptic curves of small height. Israel J. Math. 173 (2009), 253-277.
\medskip
 \bibitem[CDKS14]{3} V. Chandee, C. David, D. Koukoulopoulos and E. Smith, The Frequency of Elliptic Curve Groups Over Prime Finite Fields. Canad. J. Math. \textbf{68} (2016), no. 4, 721-761.
\medskip
\bibitem[Dav00]{12} H. Davenport, Multiplicative number theory. Third edition. Revised and with a preface by Hugh L. Montgomery. Graduate Texts in Mathematics, 74. Springer-Verlag, New York, 2000.
  \bibitem[DS13]{4} C. David and E. Smith, Elliptic curves with a given number of points over finite fields, \textit{Compositio Math.} \textbf{149} (2013), 175–203.
 \medskip
 \bibitem[DS14]{5} C. David and E. Smith, Corrigendum to “Elliptic curves with a given number of points over finite fields”,  \textit{Compositio Math.} \textbf{150} (2014), no. 8, 1347–1348.
\medskip
\bibitem[Deu41]{16} M. Deuring, Die Typen der Multiplikatorenringe elliptischer Funktionenkorpr. Abh. Math. Sem. Univ. Humbury, \textbf{14} (1941), no. 1, 197-272.
\medskip
 \bibitem[FM96]{7} E. Fouvry and M. R. Murty, On the distribution of supersingular primes, Canad. J. Math. \textbf{48} (1996), 81–104.
\medskip
\bibitem[FI2]{JFHI:2} J. Friedlander and H. Iwaniec, The divisor problem for arithmetic progressions.\textit{Acta Arith.} 45 (1985), 273-277.
\medskip
\bibitem[IK]{14} H. Iwaniec and E. Kowalski, Analytic number theory, colloquium publications, vol. 53, American Mathematical Society.
\medskip

\bibitem[Kow06]{8} E. Kowalski, Analytic problems for elliptic curves, J. Ramanujan Math. Soc. \textbf{21} (2006), 19–114.
\medskip
\bibitem[MPS14]{9}G. Martin, P. Pollack and E. Smith, Averages of the number of points on elliptic curves. Algebra Number Theory \textbf{8} (2014), no. 4, 813–836.
\medskip
\bibitem[Par15]{20} J. Parks, Amicable pairs and aliquot cycles on average. Int. J. Number Theory \textbf{11} (2015), no. 6, 1751-1790.
\medskip
\bibitem[Par16]{21} J. Parks, A remark on elliptic curves with a given number of points over finite fields. SCHOLAR—a scientific celebration highlighting open lines of arithmetic research, 165-179, Contemp. Math., \textbf{655}, Amer. Math. Soc., Providence, RI, 2015.
\medskip
\bibitem[Rom84]{13} S. Roman, The Umbral Calculus. New York: Academic Press, pp. 59-63, 1984.


\end{thebibliography}

\end{document}